\documentclass[11pt, reqno]{amsart}

\usepackage{amssymb}
\usepackage{booktabs}
\usepackage{enumerate}
\usepackage{etoolbox}
\usepackage{eucal}
\usepackage[margin=1in]{geometry}
\usepackage[colorlinks=true, allcolors=blue]{hyperref}
\usepackage[utf8]{inputenc}
\usepackage[dvipsnames]{xcolor}
\usepackage[all]{xy}

\theoremstyle{plain}
\newtheorem{theorem}{Theorem}[section]

\newtheorem{corollary}[theorem]{Corollary}
\newtheorem{lemma}[theorem]{Lemma}
\newtheorem{proposition}[theorem]{Proposition}

\theoremstyle{definition}
\newtheorem{definition}[theorem]{Definition}
\newtheorem{example}[theorem]{Example}
\newtheorem{remark}[theorem]{Remark}

\DeclareMathOperator{\Ann}{Ann}
\DeclareMathOperator{\End}{End}
\DeclareMathOperator{\Ext}{Ext}
\DeclareMathOperator{\h}{H}
\let\hom\relax 
\DeclareMathOperator{\hom}{Hom}

\DeclareMathOperator{\im}{im}
\DeclareMathOperator{\irred}{Irred}
\DeclareMathOperator{\iso}{\stackrel{\cong}{\longrightarrow}}
\DeclareMathOperator{\maxspec}{MaxSpec}
\DeclareMathOperator{\rad}{rad}
\DeclareMathOperator{\Supp}{Supp \,}

\renewcommand{\atop}[2]{\genfrac{}{}{0pt}{}{#1}{#2}}
\newcommand{\bb}[1]{\text{$\mathbb{#1}$}}
\newcommand{\bightimes}{\widehat{\bigotimes}}
\newcommand{\C}{\text{$\Bbbk$}}
\newcommand{\cal}[1]{\text{$\mathcal{#1}$}}
\newcommand{\cspan}{\textup{span}_\C}
\newcommand{\ev}[1]{\textup{ev}^*_{#1}}
\newcommand{\g}{\text{$\mathfrak{g}$}}
\newcommand{\ga}{\text{$\mathfrak{g} \otimes A$}}
\newcommand{\Gev}[1]{\textup{ev}^{\Gamma \, *}_{#1}}
\newcommand{\Gga}{\text{$(\mathfrak{g} \otimes A)^\Gamma$}}
\newcommand{\G}{\mathfrak{s}}
\newcommand{\htimes}[1][{ }]{\widehat{\otimes_{#1}}}
\newcommand{\I}{\mathfrak{i}}
\newcommand{\lie}[1]{\text{$\mathfrak{#1}$}}
\newcommand{\pra}[1][{ }]{\xrightarrow{#1}}
\newcommand{\sfm}{{\text{$\mathsf{m}$}}}
\newcommand{\sfM}{{\text{$\mathsf{M}$}}}

\numberwithin{equation}{section}

\newtoggle{details}

\iftoggle{details}{%
  \newcommand{\details}[1]{{\color{OliveGreen}\noindent\textbf{Details:}{#1}}}
                  }{
                  \newcommand{\details}[1]{}
                  }

\begin{document}

\title{Irreducible modules for equivariant map superalgebras and their extensions}

\author{Lucas Calixto}
\address{Departamento de Matem\'atica\\
		   Instituto de Ci\^encias Exatas\\
		   UFMG\\
		   Belo Horizonte, Minas Gerais, Brazil, 30.123-970}
\email{lhcalixto@mat.ufmg.br}

\author{Tiago Macedo}
\address{Department of Mathematics and Statistics\\
         University of Ottawa\\
         Ottawa, ON K1N 6N5\\
         \  and \ 
         Universidade Federal de S\~ao Paulo\\
         S\~ao Jos\'e dos Campos, S\~ao Paulo, Brazil, 12.247-014}
\email{tmacedo@unifesp.br}

\begin{abstract}
Let $\Gamma$ be a group acting on a scheme $X$ and on a Lie superalgebra $\g$, both defined over an algebraically closed field of characteristic zero $\C$.  The corresponding equivariant map superalgebra $M(\g, X)^\Gamma$ is the Lie superalgebra of equivariant regular maps from $X$ to $\g$.  In this paper we complete the classification of finite-dimensional irreducible $M(\g, X)^\Gamma$-modules when $\g$ is a finite-dimensional simple Lie superalgebra, $X$ is of finite type and $\Gamma$ is a finite abelian group acting freely on the rational points of $X$, by classifying these $M(\g,X)^\Gamma$-modules in the case where $\lie g$ is a periplectic Lie superalgebra.  We also describe extensions between irreducible modules in terms of homomorphisms and extensions between modules for certain finite-dimensional Lie superalgebras.
\end{abstract}

\maketitle
\thispagestyle{empty}

\setcounter{tocdepth}{1}
\tableofcontents

\section{Introduction}

Lie superalgebras have a wide range of applications in many areas of physics and mathematics, such as supersymmetry, string theory, conformal field theory, and number theory.  In the study of symmetry, for instance, while Lie algebras describe bosonic degrees of freedom, Lie superalgebras also allow fermionic degrees of freedom \cite{Var04}.  In number theory, affine Kac-Moody superalgebras and their representations can be used to study problems related to sums of squares and sums of triangular numbers \cite{KW94}.  For more examples, see, for instance, \cite{FL85, Ser85, GLS01, FK02}.

It is usually the case that the representation theory of Lie superalgebras is more complicated than that of their Lie algebra counterparts.  For instance, the category of finite-dimensional modules for a finite-dimensional simple Lie algebra over an algebraically closed field of characteristic zero is always semisimple, while the category of finite-dimensional modules for a finite-dimensional simple Lie \textit{super}algebra is not necessarily so.  It is thus important to describe extensions between their irreducible modules.  Despite being a subject of intense study since the birth of supersymmetry, these extensions are not known in general, and, in contrast to extensions for Lie algebras, their study is often done case by case.  See, for instance, \cite{FL84, fuks86, Pol88, SZ98, SZ99, Gru00, Gru03, BKN10, Bag12} for some of these results.

The main goal of the current paper is to develop the representation theory of certain Lie superalgebras known as equivariant map superalgebras.  Equivariant map superalgebras generalize, on the one hand, simple Lie superalgebras, and on the other hand, current and loop Lie algebras.  They are constructed in the following way.  Consider a scheme $X$, a Lie superalgebra $\g$, both defined over an algebraically closed field of characteristic zero $\C$, and a group $\Gamma$ that acts on $X$ and $\g$ by automorphisms.  The corresponding equivariant map superalgebra $M(X, \g)^\Gamma$ is the Lie superalgebra of $\Gamma$-equivariant regular maps from $X$ to $\g$.  If one denotes by $A$ the coordinate ring of $X$, then $M(X, \lie g)^\Gamma$ can be identified with the Lie subsuperalgebra $(\g \otimes_\C A)^\Gamma$ of $\g \otimes_\C A$ that consists of its $\Gamma$-fixed points under the diagonal action.

In the particular case where \lie g is a Lie algebra, $M(X, \lie g)^\Gamma$ is called equivariant map algebra.  Representations of these Lie algebras have been a subject of intense research for the last thirty years (see, for instance, the survey \cite{NS13}).  One reason is that representations of $M(\C^\times, \lie g)^\Gamma$, known as twisted loop algebras, and $M(\C, \lie g)^\Gamma$, known as twisted current algebras, are closely related to those of affine Kac-Moody Lie algebras.  In fact, when $\g$ is a finite-dimensional simple Lie algebra and $\Gamma$ is the group of automorphisms of the Dynkin diagram of $\g$, the twisted current algebra is a parabolic subalgebra of the affine Kac-Moody Lie algebra associated to $\lie g$ and $\Gamma$, and the twisted loop algebra is its centerless derived subalgebra (see \cite[Section 13.1]{kumar02}).

Finite-dimensional irreducible representations of equivariant map algebras were classified by Neher, Savage and Senesi \cite{NSS12} in the case where $\g$ is a finite-dimensional Lie algebra and $\Gamma$ is a finite group.  Finite-dimensional simple Lie superalgebras over an algebraically closed field of characteristic zero and finite-dimensional irreducible representations of the so-called basic classical Lie superalgebras were classified by Kac in \cite{kac77} and \cite{kac78}.  In \cite{savage14}, Savage classified irreducible finite-dimensional representations of equivariant map superalgebras in the case where $\lie g$ is a basic classical Lie superalgebra (or $\lie{sl}(n,n)$, $n\geq 1$, if $\Gamma$ is trivial), $X$ has a finitely-generated coordinate ring and $\Gamma$ is a finite abelian group acting freely on the rational points of $X$.  Moving beyond basic classical Lie superalgebras, the first author, Moura and Savage classified finite-dimensional irreducible representations of equivariant map queer Lie superalgebras in \cite{CMS15}.  While in the basic classical setting those irreducible representations were isomorphic to tensor products of generalized evaluation representations, in the queer case they are irreducible products of evaluation representations (see Section \ref{ss:map.superalgebras} for details).  In \cite{Bag15}, Bagci extended this classification to equivariant map superalgebras where $\lie g$ is of Cartan type.

In the current paper, we complete the classification of finite-dimensional irreducible representations of equivariant map superalgebras by describing these modules for $M(\g, X)^\Gamma$ when $\lie g={\lie p}(n)$, a periplectic Lie superalgebra, $X$ has a finitely-generated coordinate ring, $\Gamma$ is a finite abelian group acting on $\g$ and $X$, and such that the induced action of $\Gamma$ on the rational points of $X$ is free.  In Theorem~\ref{thm:main.p(n)} we prove that, similarly to the case where $\g$ is basic classical of type II, all irreducible finite-dimensional $M(\g, X)^\Gamma$-modules are evaluation modules. The technique used to prove this result is similar to the one used in \cite{savage14}.  Behind this technique are the facts that one can choose a Cartan subalgebra of ${\lie p}(n)$ that is purely even and abelian, that the root space decomposition of $\lie p(n)$ with respect to such a subalgebra is relatively similar to that of the basic case, and that ${\lie p}(n)_{\bar 0}$ is a semisimple Lie algebra.

Equipped with a complete classification of irreducible finite-dimensional $M(\g, X)^\Gamma$-modules, one can inquire about their extensions.  This is the second problem that we address in the current paper.  In Theorem~\ref{thm:main} we describe extensions between finite-dimensional irreducible $\Gga$-modules in terms of homomorphisms and extensions between finite-dimensional indecomposable (irreducible, in some cases) modules for a finite-dimensional Lie superalgebra of the form $\ga/\sfm^n$, where $\sfm$ is a maximal ideal of $A$ and $n$ is a positive integer.  The technique used here is similar to the one used in \cite{NS13} for the non-super case.  The main difference in this super setting is that one needs to describe the kernel of certain transgression maps, which we do in Proposition~\ref{prop:ker.transgression}.

\subsection*{Organization of the paper}

In Section~\ref{section:notation} we fix the notation and state some results that will be used throughout the paper.  In Section~\ref{sec:periplectic} we describe the structure of periplectic Lie superalgebras, construct and classify the finite-dimensional irreducible $M(\g, X)^\Gamma$-modules when $\g$ is of periplectic type, $X$ has a finitely-generated coordinate ring, $\Gamma$ is a finite abelian group acting on $\g$ and $X$, and such that the action of $\Gamma$ on the rational points of $X$ is free (see Section~\ref{ss:class.pn} for the main result of this section).  In Section~\ref{s:transgression} we review the construction of Lyndon-Hochschild-Serre spectral sequences in order to describe the kernel of certain transgression maps that are relevant to the computation of extensions between finite-dimensional irreducible $M(\g, X)^\Gamma$-modules (see Proposition~\ref{prop:ker.transgression}).  Using results of previous sections, in Section~\ref{Exts} we develop a general technique to describe $p$-extensions between finite-dimensional irreducible $M(\g, X)^\Gamma$-modules, and describe $1$-extensions between these modules in terms of homomorphisms and extensions between finite-dimensional indecomposable (irreducible, in some cases) modules for a finite-dimensional Lie superalgebra of the form $\ga/\sfm^n$, where $\sfm$ is a maximal ideal of $A$ and $n$ is a positive integer (see Theorem~\ref{thm:main} for the main result of this section).  We finish the paper with some examples and applications.

\subsection*{Acknowledgments}

The authors would like to thank E. Neher and A. Savage for helpful discussions and for their comments on earlier versions of this paper.   The authors would also like to thank an anonymous referee for their thorough reading and suggestions.  The first author was supported by FAPESP grant 2013/08430-4 and PRPq grant ADRC-05/2016.  The second author was supported by CNPq grant 232462/2014-3.

\subsection*{Note on the arXiv version}

For the interested reader, the tex file of the arXiv version of this paper includes hidden details of some computations, arguments and proofs that are omitted in the pdf file.  These details can be displayed by switching the {\small\texttt{details}} toggle to true in the tex file and recompiling it.

\section{Notation and Preliminaries} \label{section:notation}

\subsection{Notation}

Let $\C$ denote an algebraically closed field of characteristic zero, $\bb Z_2$ denote the finite field with two elements $\{\bar0, \bar1\}$, and $\bb Z_{>0}$ denote the set of positive integers.  All vector spaces, algebras, and tensor products will be considered over the field $\C$ (unless otherwise specified).  A vector space $V$ is said to be a super space if it is $\bb Z_2$-graded; that is, there exist subspaces $V_{\bar0}, V_{\bar1} \subseteq V$ such that $V = V_{\bar0} \oplus V_{\bar1}$.  We denote by $| \cdot |$ the degree of a homogeneous element in a super space $V$; that is, $|v| = z$ for all $v \in V_z$ and $z \in \bb Z_2$.

\subsection{Finite-dimensional Lie superalgebras} \label{ss:lie.sup}

In this subsection we follow \cite{Mus12, gav14}.  

A Lie superalgebra is a  $\bb Z_2$-graded vector space $\lie g = \lie g_{\bar0} \oplus \lie g_{\bar1}$ with a $\bb Z_2$-graded linear transformation $[\cdot , \cdot] \colon \g \otimes \g \to \g$ that satisfies $\bb Z_2$-graded versions of anticommutativity and Jacobi identity.

For any Lie superalgebra $\g = \lie g_{\bar0} \oplus \lie g_{\bar1}$, the subspace $\g_{\bar0}$ inherits the structure of a Lie algebra and the subspace $\g_{\bar1}$ inherits the structure of a $\g_{\bar 0}$-module.  A finite-dimensional simple Lie superalgebra $\lie g$ is said to be classical if the $\lie g_{\bar0}$-module $\lie g_{\bar1}$ is completely reducible. Otherwise it is said to be of Cartan type.  When $\g$ is a classical Lie superalgebra, the $\lie g_{\bar0}$-module $\lie g_{\bar1}$ is either irreducible or a direct sum of two irreducible modules. In the first case, $\lie g$ is said to be of type II, and in the second case, $\lie g$ is said to be of type I. If a classical Lie superalgebra admits an even nondegenerate invariant bilinear form, then it is said to be basic; otherwise, it is said to be strange. Table \ref{table} summarizes the classification of finite-dimensional simple Lie superalgebras with nonzero odd parts proved by V.~Kac in \cite[Theorem~5]{kac77}.

\begin{table}
\begin{tabular}{l  l}
\toprule
Lie superalgebra & Classification \\
\midrule
$A(m,n)$, \ $m > n \ge 0$ & Basic, type I \\
$A(n,n)$, \ $n \ge 1$ & Basic, type I \\
$B(m,n)$, \ $m \ge 0$, $n \ge 1$ & Basic, type II \\
$C(n+1)$, \ $n \ge 1$ & Basic, type I \\
$D(m,n)$, \ $m \ge 2$, $n \ge 1$ & Basic, type II \\
$D(2,1;\alpha)$, \ $\alpha \ne 0,-1$ & Basic, type II \\
$F(4)$ & Basic, type II \\
$G(3)$ & Basic, type II \\
$\lie p(n)$, \ $n \ge 2$ & Strange, type I \\
$\lie q(n)$, \ $n \ge 2$ & Strange, type II \\
$H(n)$, \ $n \ge 4$ & Cartan type \\
$S(n)$, \ $n \ge 3$ & Cartan type \\
$\tilde S(n)$, \ $n = 2m, m \ge 2$ & Cartan type \\
$W(n)$, \ $n \ge 2$ &  Cartan type \\
\bottomrule
\end{tabular}
\bigskip
\caption{Classification of finite-dimensional simple Lie superalgebras with nonzero odd parts.} \label{table}
\end{table}

When $\lie g$ is a classical Lie superalgebra, a Cartan subalgebra of $\lie g$ is defined to be a Cartan subalgebra of the Lie algebra $\lie g_{\bar 0}$. When $\g$ is a Lie superalgebra of Cartan type, it admits a $\bb Z$-grading $\g = \bigoplus_{-1 \le i} \g_i$ that is compatible with the $\bb Z_2$-grading; that is, $\g_{\bar 0} = \bigoplus_{i \in 2 \bb Z}\g_i$  and $\g_{\bar 1}=\bigoplus_{i \in 2\bb Z}\g_{i+1}$, and such that $\g_0$ is a reductive Lie algebra.  In this case, a Cartan subalgebra of $\lie g$ is defined to be a Cartan subalgebra of the Lie algebra $\lie g_0$.  (It is worth noting that when $\g$ is of type $\tilde S$, this is only a $\bb Z$-grading as a vector space, not as Lie superalgebras.)

Let $\g$ be a finite-dimensional simple Lie superalgebra, and let $\lie h$ be a Cartan subalgebra of $\lie g$.  The action of $\lie h$ on $\g$ is diagonalizable and we have a root space decomposition
\[
\g=\lie h\oplus \bigoplus_{\alpha\in \lie h^* \setminus \{0\}}\g_{\alpha},
\quad \textup{where} \quad
\g_\alpha = \{ x \in \g \mid [h,x] = \alpha(h) x \textup{ for all } h \in \lie h\}.
\]
Let $\Delta = \{ \alpha \in \lie h^* \setminus \{ 0 \} \mid \lie g_\alpha \neq 0 \}$ denote the set of roots of \lie g and $Q$ denote the subgroup of $\lie h^*$ generated by $\Delta$.

Let \lie g be a finite-dimensional simple Lie superalgebra. Then the set of isomorphism classes of finite-dimensional irreducible \lie g-modules is in bijection with a subset $\Lambda^+$ of $\lie h^*$ (see \cite[Theorem~8]{kac77}). We denote by $\Lambda$ the subgroup of $\lie h^*$ generated by $\Lambda^+$, and by $V(\lambda)$ the irreducible $\lie g$-module corresponding to $\lambda \in \Lambda^+$, namely the unique finite-dimensional irreducible $\g$-module of highest weight $\lambda$.  Let $\lambda^* \in \Lambda^+$ denote the highest weight of the dual \lie g-module $V(\lambda)^*$.

For a Lie superalgebra $\g$ and a $\g$-module $V$, consider the space of linear endomorphisms of the vector space $V$, $\End(V)$, and define
\begin{gather*}
\End_{\lie g}(V)
= \{\varphi\in \End (V)\mid \varphi(x v)=x\varphi(v), \text{ for all }x\in \g,\ v\in V\}, \\
\End(V)_z
= \{ \phi \in \End (V) \mid \phi (v) \in V_{z+z'} \textup{ for all } v \in V_{z'}, z' \in \bb Z_2 \}, \\
\End_{\lie g}(V)_{\bar 0}=\End_{\lie g} (V)\cap \End (V)_{\bar 0} \text{ and }  \End_{\lie g}(V)_{\bar 1}=\End_{\lie g} (V)\cap \End (V)_{\bar 1}.
\end{gather*}
Let $\lie g^1$ and $\lie g^2$ be two finite-dimensional Lie superalgebras, and $V_1$ and $V_2$ be irreducible finite-dimensional modules for $\g^1$ and $\g^2$ respectively.  The ($\lie g^1 \oplus \lie g^2$)-module $V_1 \otimes V_2$ is reducible only if $\End_{\lie g^1} (V_1)_{\bar 1} \neq 0$ and $\End_{\lie g^2} (V_2)_{\bar 1} \neq 0$ (see \cite[Proposition~8.4]{Che95}).  In this case, by Schur's Lemma for Lie superalgebras, $\End_{\lie g^i} (V_i)_{\bar 1} = \C \varphi_i$ for some $\varphi_i^2 = 1$, $i \in \{1,2\}$, and
\[
\widehat{V}
= \left\{ v \in V_1 \otimes V_2 \mid \left( \sqrt{-1} \varphi_1 \otimes \varphi_2 \right) v = v \right\}
\]
is an irreducible ($\lie g^1\oplus\lie g^2$)-submodule satisfying $V_1 \otimes V_2 \cong \widehat{V}\oplus \widehat{V}$ (see \cite[p.~27]{Che95}).  Define the irreducible product $V_1 \htimes V_2$ to be
\[
V_1 \htimes V_2 =
\begin{cases}
V_1\otimes V_2,
& \textup{if $V_1 \otimes V_2$ is irreducible}, \\
\widehat V,
& \text{otherwise}.
\end{cases}
\]
If $\lie g^1$ and $\lie g^2$ are finite-dimensional simple Lie superalgebras not of type $\lie q$, then the irreducible product is always equal to the tensor product.

Given $\ell>1$, finite-dimensional Lie superalgebras $\lie g^1, \dotsc, \lie g^\ell$, and for each $i \in \{1, \dotsc, \ell\}$, an irreducible finite-dimensional $\lie g^i$-module $V_i$, define the $(\lie g^1 \oplus \dotsb \oplus \lie g^\ell)$-module $V_1 \htimes \dotsb \htimes V_\ell$ to be
\[
V_1 \htimes \dotsb \htimes V_\ell
= (V_1 \htimes \dotsb \htimes V_{\ell-1}) \htimes V_\ell.
\]
Up to isomorphism, $\htimes$ is associative and, when $\lie g^1 = \dots = \lie g^\ell$, $\htimes$ is also commutative (see \cite[Lemma~6.2]{CMS15}).  Also, define
\begin{equation} \label{eq:kappa.mods}
\kappa(V_1, \ldots, V_{\ell})
= \sum_{i=2}^{\ell} \dim \End_{\left( \g^1 \oplus \dotsb \oplus \g^{i-1} \right)} \left( V_1 \htimes \dotsb \htimes V_{i-1} \right)_{\bar 1} \dim \End_{\lie g^i} (V_i)_{\bar1}.
\end{equation}
One can prove by induction that $\kappa (V_1, \dotsc, V_\ell) \le \ell-1$ and that
\[
V_1 \otimes \dotsb \otimes V_\ell 
\cong (V_1 \htimes \dotsb \htimes V_\ell)^{\oplus 2^{\kappa (V_1, \dotsc, V_\ell)}}.
\]

\subsection{Equivariant map superalgebras} \label{ss:map.superalgebras}

In this subsection, we review results proved in \cite{savage14, Bag15, CMS15}. 

Let $\lie g$ be a finite-dimensional simple Lie superalgebra and $A$ be an associative, commutative, finitely-generated $\C$-algebra with unit.  The map superalgebra $\ga$ is the Lie superalgebra with underlying vector space $\lie g \otimes_\C A$, with $\bb Z_2$-grading given by $(\lie g \otimes A)_z = \lie g_z \otimes A$, $z \in \bb Z_2$, and with Lie superbracket extending bilinearly
\[
[x \otimes a, y \otimes b] = [x,y] \otimes ab,
\quad \textup{for all } x,y \in \g \textup{ and } a,b \in A.
\]
Let $\Gamma$ be a group acting on $\g$ and $A$ by automorphisms.  One can induce an action of $\Gamma$ on $\ga$ by extending linearly
\[
\gamma (x \otimes a) = \gamma(x) \otimes \gamma (a),
\quad \textup{for all } \gamma \in \Gamma, x \in \g \textup{ and } a \in A.
\]
The equivariant map superalgebra $\Gga$ is the Lie subsuperalgebra of $\ga$ consisting of its $\Gamma$-fixed points:
\[
\Gga = \{ x \in \ga \mid \gamma x = x \textup{ for all } \gamma \in \Gamma \}.
\]

Let $\maxspec (A)$ denote the set of maximal ideals of $A$. Notice that the action of $\Gamma$ on $A$ induces an action of $\Gamma$ on $\maxspec(A)$ that is explicitly given by $\gamma \sfm = \{ \gamma a \mid a \in \sfm \} \in \maxspec(A)$ for all $\sfm \in \maxspec(A)$ and $\gamma \in \Gamma$.  For most finite-dimensional simple Lie superalgebras $\g$, if $\Gamma$ is a finite abelian group acting freely on $\maxspec (A)$, then every finite-dimensional irreducible $\Gga$-module can be described in terms of generalized evaluation modules.  These generalized evaluation modules are defined in the following way.  Given $\sfm \in \maxspec(A)$ and $n \in \bb Z_{>0}$, define ${\rm ev}_{\sfm^n}$ to be the homomorphism of Lie superalgebras given by the composition
\[
{\rm ev}_{\sfm^n} 
\colon \ga 
\to \ga / \lie g \otimes \sfm^n
\iso \ga / \sfm^n.
\]

\begin{center}\textit{
For the rest of this subsection, assume that $\Gamma$ is a finite group acting freely on $\maxspec (A)$.
}\end{center}

In this case, the restriction of ${\rm ev}_{\sfm^n}$ to $\Gga$ is surjective (see \cite[Lemma~5.6]{savage14}), and induces a surjective homomorphism of Lie superalgebras
\[
{\rm ev}_{\sfm^n}^\Gamma \colon \Gga \to \ga / \sfm^n.
\]
Given a $\ga / \sfm^n$-module $V$ with associated representation $\rho \colon \ga / \sfm^n \to \lie{gl}(V)$, define $\Gev{\sfm^n} (V)$ to be the $\Gga$-module with associated representation given by the pull-back of $\rho$ along ${\rm ev}_{\sfm^n}^\Gamma$,
\[
\Gev{\sfm^n}(\rho) \colon
\Gga \pra[{\rm ev}_{\sfm^n}^\Gamma]
\ga/\sfm^n \pra[\ \rho \ ]
\lie{gl} (V).
\]

Given $\sfm \in \maxspec(A)$ and $n \in \bb Z_{>0}$, denote by ${\rm Irred}(\ga/\sfm^n)$ the set of isomorphism classes of finite-dimensional irreducible $\ga/\sfm^n$-modules, and denote by $\cal R$ the disjoint union
\[
\cal R =
\bigsqcup_{\atop{n \in \bb Z_{>0}}{\sfm \in \maxspec(A)}} {\rm Irred}(\ga/\sfm^n).
\]
Notice that the action of $\Gamma$ on $\ga$ induces an action of $\Gamma$ on $\cal R$.  Namely, let $[V]$ be an element in $\cal R$, and let $V$ be a representative of the class $[V]$ with associated representation $\rho \colon \ga / \sfm^n \to \lie{gl}(V)$. For each $\gamma \in \Gamma$, define $\gamma\cdot [V]$ in $\cal R$ to be the isomorphism class of the $\ga / (\gamma\sfm)^n$-module $V^\gamma$, whose associated representation $\rho^\gamma \colon \ga / (\gamma\sfm)^n \to \lie{gl}(V)$ is given by $\rho^\gamma (x) = \rho (\gamma^{-1} x)$, for all $x \in \ga / (\gamma\sfm)^n$.

Denote by $\cal P$ the set of $\Gamma$-equivariant functions $\pi \colon \maxspec(A) \to \cal R$ satisfying the following conditions:
\begin{enumerate}[$\bullet$]
\item For each $\sfm \in \maxspec(A)$, $\pi (\sfm) \in {\rm Irred}(\ga/\sfm^n)$ for some $n > 0$;

\item $\pi(\sfm)$ is the isomorphism class of the trivial (one-dimensional irreducible) module for all but finitely many $\sfm \in \maxspec(A)$.
\end{enumerate}
Notice that for any two representatives $V$ and $V'$ of $\pi(\sfm)\in {\rm Irred} (\ga/\sfm^n)$, there is an isomorphism of $\Gga$-modules $\Gev{\sfm^n} V \cong \Gev{\sfm^n} V'$.  Thus we will abuse notation and for each maximal ideal $\sfm \subseteq A$, we will denote by $\pi(\sfm)$ an arbitrary but fixed choice of $(\ga/\sfm^n)$-module representative of $\pi(\sfm)$.

Given $\pi \in \cal P$, define its support to be $\Supp(\pi) = \{ \sfm \in \maxspec(A) \mid \pi(\sfm) \textup{ is nontrivial} \}$, and let $\Supp_*(\pi)$ be a subset of $\maxspec(A)$ which contains exactly one element of each $\Gamma$-orbit in $\Supp(\pi)$. Since every $\pi \in \cal P$ is $\Gamma$-equivariant, up to isomorphism, the $\Gga$-module $\bightimes_{\sfm \in \Supp_* (\pi)} \Gev{\sfm^n} \pi(\sfm)$ is independent of the choice of $\Supp_* (\pi)$ (see \cite[Lemma~5.9]{savage14}).  Thus for every $\pi \in \cal P$, we fix an arbitrary subset $\Supp_*(\pi)$ as above and define $\cal V(\pi)$ to be the $\Gga$-module
\[
\cal V(\pi) =
\widehat{\bigotimes_{\sfm \in \Supp_*(\pi)}} \Gev{\sfm^n}\pi(\sfm).
\]

For most finite-dimensional simple Lie superalgebras, it is known that every finite-dimensional irreducible $\Gga$-module is isomorphic to $\cal V(\pi)$ for a unique $\pi \in \cal P$.  This was proved by Savage when $\g$ is a basic Lie superalgebra (see \cite[\textsection 7]{savage14}), by Bagci when $\g$ is a Lie superalgebra of Cartan type (see \cite[Theorem 4.3]{Bag15}), and by the first author, Moura and Savage when $\g$ is a queer Lie superalgebra (see \cite[Theorem~7.1]{CMS15}).  In Section \ref{sec:periplectic}, we will show that, if $\g$ is of type $\lie p$, then every finite-dimensional irreducible $\Gga$-module is also isomorphic to $\cal V(\pi)$ for a unique $\pi \in \cal P$, and that $n = 1$ for all $\sfm \in \Supp_*(\pi)$.  This completes the classification of finite-dimensional irreducible modules for equivariant map superalgebras associated to finite-dimensional simple Lie algebras.

Moreover, if $\lie g$ is of type II, $H$, $S$ or $\tilde S$, then every finite-dimensional irreducible $\Gga$-module is isomorphic to $\cal V(\pi)$ for a unique $\pi \in \cal P$ with $n = 1$ for all $\sfm \in \Supp_*(\pi)$ (these modules are called evaluation modules by Savage, see \cite[Definition~5.2]{savage14}).  It is important to remark that if $\lie g$ is of type I, then there exist finite-dimensional irreducible $\Gga$-modules which are isomorphic to generalized evaluation modules but not to evaluation modules (see \cite[\textsection 4]{rao13}).

\begin{remark} \label{rmk:gather.gen.ev}
Let $n, n' \in \bb Z_{> 0}$ with $n' \le n$, let $\sfm$ be a maximal ideal in $A$, and let $V$ be a $\ga/\sfm^{n'}$-module with corresponding representation $\rho \colon \ga/\sfm^{n'} \to \lie{gl} (V)$.  Since $n' \le n$, then $\sfm^n \subseteq \sfm^{n'}$, and therefore we have a canonical projection $\pi \colon \ga/\sfm^n \twoheadrightarrow \ga/\sfm^{n'}$.  We can thus regard $V$ as a $\ga/\sfm^n$-module with representation given by $\rho \circ \pi$.

Notice that as representations of $\Gga$, $\Gev{\sfm^{n'}} (\rho \circ \pi)$ and $\Gev{\sfm^{n}} (\rho)$ are the same.  Hence, the $\Gga$-modules $\Gev{\sfm^{n'}} V$ and $\Gev{\sfm^{n}} V$ are isomorphic.  Moreover, if $\rho$ is an irreducible representation of $\ga/\sfm^{n'}$, then $\rho \circ \pi$ is an irreducible representation of $\ga/\sfm^{n}$.  As a consequence, given any two finite-dimensional irreducible $\Gga$-modules $V$ and $V'$, one loses no generality in assuming that there exist integers $\ell, n \in \bb Z_{>0}$, maximal ideals $\sfm_1, \dotsc, \sfm_\ell \in \maxspec(A)$ in distinct $\Gamma$-orbits, and for each $i \in \{ 1, \dotsc, \ell \}$, irreducible $\ga/\sfm_i^n$-modules $V_i, V'_i$, such that $V \cong \bightimes_{i=1}^\ell \Gev{\sfm_i^n} V_i$ and $V' \cong \bightimes_{i=1}^\ell \Gev{\sfm_i^n} V'_i$.
\end{remark}

Given $\pi \in \cal P$, let $\ell, n_1, \dotsc, n_\ell \in \bb Z_{>0}$, $\sfm_1, \dotsc, \sfm_\ell$ be maximal ideals of $A$ in distinct $\Gamma$-orbits, and for each $i \in \{ 1, \dotsc, \ell \}$, let $V_i$ be an irreducible $\ga/\sfm_i^{n_i}$-module such that $\cal V (\pi) \cong \bightimes_{i=1}^\ell \Gev{\sfm_i^{n_i}} V_i$.  Recall from \eqref{eq:kappa.mods} the definition of $\kappa(V_1, \dotsc, V_\ell)$, and define $\kappa (\pi)$ to be
\begin{equation}\label{eq:kappa.function}
\kappa(\pi) = \kappa (V_1, \dotsc, V_\ell),
\end{equation}
and notice that $\cal V(\pi)^{\oplus 2^{\kappa(\pi)}} \cong \bigotimes_{\sfm \in \Supp(\pi)} \Gev{\sfm^n} \pi (\sfm)$ for all $\pi \in \cal P$.

\subsection{Ideals}\label{subsec:ideals}

Throughout this subsection, let $\lie g$ be a finite-dimensional simple Lie superalgebra and $A$ be an associative, commutative, finitely-generated $\C$-algebra with unit.  Define the support of an ideal $I \subseteq A$ to be
\[
\Supp (I) = \{ \sfm \in \maxspec A \mid I \subseteq \sfm\}.
\]

\begin{lemma} \label{lem:assoc-alg-facts}
Let $I$ and $J$ be ideals of $A$.
\begin{enumerate}[$(a)$]
\item \label{lem-item:support-power} For any $n>0$, we have $\Supp(I)=\Supp(I^n)$.

\item \label{lem-item:finiteCod-and-finiteSupp} If $A$ is finitely generated, then $\Supp(I)$ is finite if and only if $I$ has finite codimension in $A$.

\item \label{lem-item:product-intersection} If $\Supp(I) \cap \Supp(J) = \varnothing$, then $I+J=A$ and $IJ = I \cap J$.  Moreover, $I^m + J^n = A$ and $I^m J^n = I^m \cap J^n$ for any $m, n > 0$.

\item \label{lem-item:Noetherian-power-radical} If $A$ is Noetherian, then every ideal $I \subseteq A$ contains a power of its radical. In particular, $\rad I = \prod_{\sfm \in \Supp (I)} \sfm$.
\end{enumerate}
\end{lemma}

\begin{proof}
To prove part \eqref{lem-item:support-power}, fix $n > 0$.  It is clear that $\Supp(I) \subseteq \Supp (I^n)$.  The reverse inclusion follows from the fact that maximal ideals are also prime. Indeed, suppose $\sfm$ is a maximal ideal containing $I^n$.  Then for all $a \in I$, we have $a^n \in \sfm$.  Since maximal ideals are prime ideals, this implies that $a \in \sfm$, showing that $I \subseteq \sfm$.  The proofs of parts \eqref{lem-item:finiteCod-and-finiteSupp}, \eqref{lem-item:product-intersection} and \eqref{lem-item:Noetherian-power-radical} can be found in \cite[\S 2.1]{savage14}.
\end{proof}

The next result is a generalization of \cite[Corollary~4.17]{savage14} and will be used in the proof of Lemma~\ref{lem:h1(ga,M)}.

\begin{proposition} \label{prop:ann.fin.dim}
Let $\Gamma$ be a finite group acting freely on $\maxspec (A)$ and $M$ be a finite-di\-men\-sion\-al $(\ga)^\Gamma$-module. Then there exist $\ell, n \in \bb Z_{>0}$ and $\sfm_1, \dots, \sfm_\ell \in \maxspec(A)$, such that $\sfm_1^n \cdots \sfm_\ell^n$ is a $\Gamma$-invariant finite-codimensional ideal of $A$ and $(\lie g \otimes \sfm_1^n \cdots \sfm_\ell^n)^\Gamma M = 0$.
\end{proposition}

\begin{proof}
We begin by proving the case where $\Gamma$ is trivial.  Let $\rho \colon \lie g \otimes A \to \lie{gl} (M)$ be the representation of $\ga$ corresponding to $M$. Notice that $\ker\rho$ is a finite-codimensional ideal of $\ga$, since $M$ is finite dimensional. Thus, by \cite[Proposition~8.1]{savage14}, $\ker\rho$ must be of the form $\g\otimes I$, for some finite-codimensional ideal $I \subseteq A$. Now, by Lemma~\ref{lem:assoc-alg-facts}\eqref{lem-item:finiteCod-and-finiteSupp}, the fact that $I$ is finite-codimensional implies that $I$ has finite support; that is, $\Supp(I) = \{ \sfm_1, \ldots, \sfm_\ell \}$ for some $\ell > 0$.  Since these maximal ideals are pairwise distinct, the radical of $I$ is given by $\rad I = \sfm_1 \cdots \sfm_\ell$.  Moreover, since $A$ is assumed to be finitely generated, by Lemma~\ref{lem:assoc-alg-facts}\eqref{lem-item:Noetherian-power-radical}, there exists $n > 0$, such that $(\rad I)^n \subseteq I$; that is, such that $\sfm_1^n \cdots \sfm_\ell^n \subseteq I$.  We thus conclude that there exist $n, \ell \in \bb Z_{>0}$ and $\sfm_1, \dots, \sfm_\ell \in \maxspec(A)$ such that $(\lie g \otimes \sfm_1^n \cdots \sfm_\ell^n) M = 0$.

For the case where $\Gamma$ acts non trivially on M, we recall that, by \cite[Proposition~8.5]{savage14}, $M$ is the restriction of a finite-dimensional $\ga$-module $M'$.  From the case where $\Gamma$ is trivial, since $M'$ is finite dimensional, there exist $n, k \in \bb Z_{>0}$ and maximal ideals $\sfm_1, \dots, \sfm_k \subseteq A$, such that $(\lie g \otimes \sfm_1^n \cdots \sfm_k^n) M = 0$.  Consider all ideals of the form $\gamma \sfm_i$, with $\gamma \in \Gamma$ and $i \in \{1, \dotsc, k\}$.  Since $\Gamma$ is a finite group, we can enumerate the elements of $\{ \gamma\sfm_i \mid \gamma \in \Gamma, i \in \{1, \dotsc, k\} \}$ as $\sfm_1, \dotsc, \sfm_\ell$, with $\ell \le |\Gamma| k$.  (Notice that $\ell \le |\Gamma| k$ only if $\sfm_i \notin \Gamma \sfm_j$, for $1 \le i \ne j \le k$.)  Moreover, notice that $\sfm_1^n \cdots \sfm_\ell^n$ is a $\Gamma$-invariant finite-codimensional ideal and
\[
(\g\otimes\sfm_1^n \cdots \sfm_\ell^n)^\Gamma M
= (\g\otimes\sfm_1^n \cdots \sfm_k^n)^\Gamma M'
\subseteq (\g\otimes\sfm_1^n \cdots \sfm_k^n)M'
=0.
\qedhere
\]
\end{proof}

\subsection{Cohomology of Lie superalgebras} \label{ss:homology.cohomology}

Let $\lie g = \lie g_{\bar0} \oplus \lie g_{\bar1}$ be a Lie superalgebra and $V, U$ be $\lie g$-modules.  As usual, the $n$-th extension group $\Ext^n_{\lie g}(V, U)$ denotes the $n$-th right derived functor of $\hom_{U(\lie g)}(V, -)$ applied to $U$.  In particular, the cohomology of $\lie g$ with coefficients in $U$, is $\h^\bullet (\g, U) = \Ext^\bullet_\lie g (\C, U)$.

In order to compute $\Ext^n_{U(\lie g)}(V,U)$, first notice that $\Lambda^n \lie g \otimes V$ and $\hom_\C (\Lambda^n \lie g, V)$ are naturally $\bb Z_2$-graded. Namely, for all $n \ge 0$ and $z \in \bb Z_2$,
\begin{gather*}
(\Lambda^n \lie g \otimes V)_z = \{ x_1 \wedge \dotsb \wedge x_n \otimes v \mid |x_1| + \dotsb + |x_n| + |v| = z \}, \\
\hom_\C (\Lambda^n \lie g, V)_z = \{ f \colon \Lambda^n \lie g \to V \mid f(\lambda) \in V_{{w+z}}, \ \lambda \in (\Lambda^n\lie g)_{w}, w \in \bb Z_2 \}.
\end{gather*}
Then consider $\hom_\C (V, U)$ as a $\g$-module, where, for homogeneous elements $f \in \hom_\C (V, U)$ and $x\in \g$, we have
	\[
(xf)(v)=x (f(v))-(-1)^{|x||f|} f(x v) \quad \text{ for all } v \in V.
	\]
Finally, identify $\Ext^\bullet_\lie g (V,U)$ with the cohomology of the cocomplex 
\begin{equation} \label{chevalley.cocomplex}
0 \pra \hom_\C (V, U) \pra[\partial^0] \hom_\C (\g, \hom_\C(V, U)) \pra[\partial^1] \hom_{\C} (\Lambda^2 \g , \hom_\C(V, U)) \pra[\partial^2] \dotsb ,
\end{equation}
where, for all $f \in \hom_\C (\Lambda^n \g, \hom_\C(V, U))$, its image under the differential $\partial^n$ extends linearly
\begin{align} \label{eq:codiffs}
\partial^nf & (x_0 \wedge \dots \wedge x_n) \notag \\
&= \sum_{0 \le i \le n} (-1)^{i+|x_i|(|f|+|x_0|+ \cdots + |x_{i-1}|)}x_i f(x_0 \wedge \dots \wedge \widehat{x_i} \wedge \dots \wedge x_n) \\
&{\quad}+  \sum_{0 \le i < j \le n} (-1)^{j+|x_j|(|x_{i+1}|+ \cdots + |x_{j-1}|)} f(x_0 \wedge \dots x_{i-1}\wedge [x_i, x_j] \wedge x_{i+1}\wedge\dots \wedge \widehat{x_j} \wedge \dots \wedge x_n), \notag
\end{align}
for all homogeneous elements $x_0, \dotsc, x_n \in \g$, $n \ge 0$.  Notice that, the differentials $\partial^\bullet$ respect the $\bb Z_2$-grading on $\hom_\C (\Lambda^\bullet \g, \hom_\C(V, U))$, thus it induces a $\bb Z_2$-grading on $\Ext^\bullet_\g (V, U)$.  We denote $\h^\bullet (\g, U)_z$ by $\h^\bullet_z (\g, U)$ for all $z \in \bb Z_2$.

We now state some homological techniques that will be used in this paper.  This first lemma reduces the computation of extensions between finite-dimensional modules to the computation of cohomologies.

\begin{lemma} \label{lem:isos}
If $\lie s$ is a Lie superalgebra and $U$, $V$ and $W$ are finite-dimensional $\lie s$-modules, then we have the following isomorphisms of $\lie s$-modules:
\begin{gather*}
U \otimes V \cong V \otimes U, \quad (U\otimes V)^* \cong U^*\otimes V^*, \\
\hom_\C (U \otimes V, W) \cong \hom_\C (U, V^*\otimes  W), \textup{ and} \\
\Ext^n_{\lie s} (V,U) \cong \h^n (\lie s, V^* \otimes U), \ \textup{for all } n>0.
\end{gather*}
\end{lemma}
\begin{proof}
The proofs are similar to those in \cite[Lemma 3.1.13]{kumar02}.
\end{proof}

The following lemma reduces the computation of the cohomology of any trivial $\lie s$-module; that is, any $\lie s$-module $V$ such that $\lie s \cdot V=0$, to that of the trivial $\lie s$-module $\C$.

\begin{lemma} \label{lem:triv.mods}
If $\lie s$ is a Lie superalgebra and $V$ is a finite-dimensional trivial $\lie s$-module, then
\[
\h^\bullet (\lie s, V) \cong \h^\bullet(\lie s, \C) \otimes V.
\]
\end{lemma}
\begin{proof}
This proof follows directly from the definition of the differentials \eqref{eq:codiffs}.
\end{proof}

The following proposition is a special case of the well known K\"unneth formula.

\begin{proposition} \label{prop:kunneth}
If $\lie s^1, \lie s^2$ are Lie superalgebras, $U_1, V_1$ are $\lie s^1$-modules and $U_2, V_2$ are $\lie s^2$-modules, then
\[
\Ext^n_{\lie s^1 \oplus \lie s^2} (U_1 \otimes U_2, V_1 \otimes V_2) \cong 
\bigoplus_{p+q=n} \Ext^p_{\lie s^1} (U_1, V_1) \otimes \Ext^q_{\lie s^2} (U_2, V_2),
\quad n \ge 0.
\]
\end{proposition}
\begin{proof}
The proof follows from \cite[Theorem 3.6.3]{weibel94} using standard arguments.
\end{proof}

The following result is a graded version of Lyndon-Hochschild-Serre spectral sequence. In the superalgebra setting, a filtration by powers of an ideal turns out to be $\bb Z_2$-graded, thus yielding two spectral sequences that converge respectively to even and odd cohomologies (see Section~\ref{s:transgression} for further details).

\begin{proposition} \label{prop:lhsss}
If $\lie s$ is a Lie superalgebra, $V$ is a $\lie s$-module and $\lie i \subseteq \lie s$ is an ideal, then there exist first-quadrant cohomology convergent spectral sequences
\[
E_{2}^{p,q} \cong \h^p_z (\lie s/\lie i, \h^q_z (\lie i, V)) \Rightarrow \h^{p+q}_z (\lie s, V), \quad z \in \bb Z_2.
\]
\end{proposition}
\begin{proof}
This proof is given in \cite[Chapter 1, \textsection 6.5]{fuks86} and \cite[Theorem~16.6.6]{Mus12}.
\end{proof}

The following result is a superalgebra generalization of a well-known result for Lie algebra cohomology.  We record it as it will be used to prove Lemma~\ref{lem:com.even}.

\begin{lemma} \label{lem:h1aC}
For any Lie superalgebra $\lie s$, we have
\[
\h^1_{\bar 0} (\lie s, \C) \cong (\lie s_{\bar 0} / ([\lie s_{\bar 0}, \lie s_{\bar 0}] + [\lie s_{\bar 1}, \lie s_{\bar 1}]))^*
\quad \textup{and} \quad
\h^1_{\bar 1} (\lie s, \C) \cong \h^0 (\lie s_{\bar 0}, \lie s_{\bar 1}^*) \cong (\lie s_{\bar 1} / [\lie s_{\bar 0}, \lie s_{\bar 1}])^*.
\]
\end{lemma}
\begin{proof}
Recall that the cocomplex $\Lambda^\bullet \lie s^*$ is $\bb Z_2$-graded, that the differential $\partial^\bullet$ preserves this grading, and that it induces a $\bb Z_2$-grading on $\h^\bullet (\lie s, \C)$.  We will compute each graded component of $\h^1 (\lie s, \C)$.  First notice that the restriction of $\partial^1$ to the even part is $\partial^1_{\bar 0} \colon \lie s_{\bar 0}^* \to \Lambda^2 \lie s_{\bar 0}^* \oplus S^2 \lie s_{\bar 1}^*$, and that the restriction of $\partial^1$ to the odd part is $\partial^1_{\bar 1} \colon \lie s_{\bar 1}^* \to \lie s_{\bar 0}^* \otimes \lie s_{\bar 1}^*$.

By definition, $\h^1 (\lie s, \C) = \ker (\partial^1) / \im (\partial^0)$, where $\partial^0 \colon \C \to \lie s^*$ is given by $\partial^0 (\lambda) (x) = x \cdot \lambda = 0$ for all $\lambda \in \C, x \in \lie s$, and $\partial^1 \colon \lie s^* \to \Lambda^2 \lie s^*$ is given by $\partial^1 (\varphi)(x \wedge y) = - \varphi ([x,y])$ for all $\varphi \in \lie s^*, x,y \in \lie s$.  So, in order to compute $\h^1 (\lie s, \C)$ it is enough to determine $\ker (\partial^1)$.

From the formula of $\partial^1$, it follows that $\partial^1_{\bar 0} (\varphi) = 0$ if and only if $\varphi ([\lie s_{\bar 0}, \lie s_{\bar 0}] + [\lie s_{\bar 1}, \lie s_{\bar 1}]) = 0$.  Thus $\h^1_{\bar 0} (\lie s, \C) \cong (\lie s_{\bar 0} / [\lie s_{\bar 0}, \lie s_{\bar 0}] + [\lie s_{\bar 1}, \lie s_{\bar 1}])^*$.  Also from the formula of $\partial^1$, one can see that $\ker(\partial^1_{\bar 1})$ is the kernel of $\partial^0$ in the cocomplex for computing the cohomology of the Lie algebra $\lie s_{\bar 0}$ with coefficients in $\lie s_{\bar 1}^*$  (see Section \ref{ss:homology.cohomology}).  Thus $\h^1_{\bar 1} (\lie s, \C) \cong \h^0 (\lie s_{\bar 0}, \lie s_{\bar 1}^*)$.
\end{proof}

The next result follows from the previous one.  It will be used in the proof of Lemma~\ref{lem:h1(ga,M)}.

\begin{lemma} \label{lem:com.even}
Let $\lie g$ be a finite-dimensional simple Lie superalgebra, $A$ be an associative, commutative algebra with unit, $\Gamma$ be an abelian group acting on $\g$ and $A$ by automorphisms, and $I$ be a $\Gamma$-invariant ideal of $A$.  If $\lie g_{\bar 1}$ is nonzero, then
\[
\h_{\bar0}^{1} ( (\lie g \otimes I)^\Gamma, \C ) \cong \left(\left(\lie g_{\bar 0} \otimes I/I^2\right)^\Gamma \right)^*
\quad \textup{and} \quad
\h_{\bar1}^{1} ( (\lie g \otimes I)^\Gamma, \C) \cong \left(\left(\lie g_{\bar 1} \otimes I/I^2\right)^\Gamma\right)^*.
\]
\end{lemma}
\begin{proof}
This proof follows from the fact that $\lie g$ is a simple Lie algebra and Lemma~\ref{lem:h1aC}.
\details{
First notice that $\lie g_{\bar 0}\oplus [\lie g_{\bar 0}, \lie g_{\bar 1}]$ and $[\lie g_{\bar 1}, \lie g_{\bar 1}]\oplus \lie g_{\bar 1}$ are ideals of $\lie g$.  When $\g$ is simple and $\g_{\bar1}$ is nonzero, this implies that $[\g_{\bar0}, \g_{\bar1}] = \g_{\bar1}$ and $[\g_{\bar1}, \g_{\bar1}] = \g_{\bar0}$.  Hence:
\[
\left[ (\lie g_z \otimes I)^\Gamma, (\lie g_{\bar1} \otimes I)^\Gamma \right]
= \left[ \lie g_z \otimes I, \lie g_{\bar1} \otimes I \right]^\Gamma
= \left( [\lie g_z, \lie g_{\bar1}] \otimes I^2 \right)^\Gamma
= (\lie g_{z+\bar1} \otimes I^2)^\Gamma,
\quad z \in \bb Z_2.
\]
The result follows from Lemma~\ref{lem:h1aC}.
}
\end{proof}

The next corollary follows from Lemma~\ref{lem:com.even}.

\begin{corollary} \label{cor:h1(ga,C)=0}
Let $\lie g$ be a finite-dimensional simple Lie superalgebra with $\lie g_{\bar 1} \neq 0$, $A$ be an associative, commutative algebra with unit, and $\Gamma$ be an abelian group acting on $\g$ and $A$ by automorphisms.  If $M$ and $N$ are finite-dimensional trivial $\Gga$-modules, then $\Ext^1_\Gga (M, N) = 0$.  In particular, $\h^1 (\Gga, \C) = 0$.
\end{corollary}
\begin{proof}
This proof follows directly from Lemmas~\ref{lem:isos}, \ref{lem:triv.mods} and \ref{lem:com.even}.
\details{
By Lemma~\ref{lem:isos}, $\Ext^1_\Gga (M, N) \cong \h^1 (\Gga, M^* \otimes N)$.  Since $M$ and $N$ are finite-dimensional trivial $\Gga$-modules, by Lemma~\ref{lem:triv.mods},
\[
\h^1 (\Gga, M^* \otimes N) \cong \h^1 (\Gga, \C) \otimes (M^* \otimes N). 
\]
Since $A$ is an associative, commutative algebra with unit, then $A^2 = A$.  Thus, by Lemma~\ref{lem:com.even}, $\h^1 (\Gga, \C) = 0$.  The result follows.
}
\end{proof}

\section{Equivariant periplectic map Lie superalgebras} \label{sec:periplectic}

\subsection{Structure of periplectic Lie superalgebras}

Given $n\geq 2$, the periplectic Lie superalgebra $\lie p(n)$ is the Lie subalgebra of $\lie{gl}(n+1, n+1)$ whose elements are matrices of the form
\begin{equation}\label{eq:matrixper}
  M=\left(\begin{array}{r|r}
    A & B  \\
    \hline
    C & -A^t \\
  \end{array}\right),
\end{equation}
where $A \in \lie{sl}_{n+1}$, $B=B^t$ and $C^t=-C$.

\begin{center}\textit{
Throughout this section, we will denote $\lie p(n)$ by $\g$.
}\end{center}

The even part, $\g_{\bar0}$, is isomorphic to the Lie algebra $\lie{sl}_{n+1}$.  The structure of $\g_{\bar 1}$ is the following.  Let $S^2 (\C^{n+1})$ (resp. $\Lambda^2 (\C^{n+1})^*$) denote the second symmetric (resp. exterior) power of $\C^{n+1}$ (resp. $(\C^{n+1})^*$), with the action of $\lie{sl}_{n+1}$ induced by matrix multiplication, and let $\g_{\bar 1}^+$ (resp. $\g_{\bar 1}^-$) be the set of all matrices of the form \eqref{eq:matrixper} such that $A=C=0$ (resp. $A=B=0$).  As a $\g_{\bar 0}$-module, we have: $\g_{\bar 1} \cong \g_{\bar 1}^+ \oplus \g_{\bar 1}^-$, $\g_{\bar 1}^+ \cong S^2 (\C^{n+1})$ and $\g_{\bar 1}^- \cong \Lambda^2 (\C^{n+1})^*$.

If we set $\g_{-1} = \g_{\bar 1}^-$, $\g_0 = \g_{\bar 0}$ and $\g_1 = \g_{\bar 1}^+$, then $\g= \g_{-1} \oplus \g_0 \oplus \g_1$ is a $\bb Z$-grading of $\g$ which is compatible with the $\bb Z_2$-grading; that is, $\g_{\bar 0} = \g_0$ and $\g_{\bar 1} = \g_{-1} \oplus \g_1$.  Let $\lie h \subseteq \g_0$ be the Cartan subalgebra of $\g_0$ formed by diagonal matrices.  Since $\g_0$ is isomorphic to $\lie{sl}_{n+1}$, we can identify $\lie h$ with its dual via the bilinear, nondegenerate, $\g_0$-invariant form $(A_1, A_2) = {\rm tr} (A_1 A_2)$.  For $i \in \{ 1, \dotsc, n+1 \}$, let $\varepsilon_i$ be the unique element in $\lie h^*$ satisfying
\[
\varepsilon_i (E_{j,j} + E_{k,k} - E_{n+j+1, n+j+1} - E_{n+k+1, n+k+1}) = \delta_{i,j} + \delta_{i,k}
\qquad \textup{for all $1 \le j \ne k \le n+1$}.
\]
 The roots of $\g$ are described as follows:
\begin{enumerate}[$\bullet$]
\item Roots of $\g_{-1}$: $-\varepsilon_i-\varepsilon_j$, where $1 \leq i < j\le n+1$.

\item Roots of $\g_0$: $\varepsilon_i-\varepsilon_j$, where $1 \le i \ne j \le n+1$.

\item Roots of $\g_1$: $\varepsilon_i+\varepsilon_j$, where $1 \le i \le j \le n+1$.
\end{enumerate}

Choose a triangular decomposition $\lie n_0^-\oplus \lie h\oplus\lie n_0^+$ of the Lie algebra $\g_0$ such that the positive roots are: $\varepsilon_i-\varepsilon_j$, with $i<j$.  We fix a triangular decomposition $\lie n^- \oplus \lie h \oplus \lie n^+$ of $\lie p(n)$, where $\lie n^\pm = \g_{\pm 1} \oplus \lie n_0^\pm$, and denote by $\lie b = \lie h \oplus \lie n^+$ a Borel subalgebra of $\lie p(n)$.  Notice that all the roots of $\g_1$ are positive.

\subsection{Construction of finite-dimensional irreducible modules}

Recall that throughout this section we are denoting $\g = \lie p(n)$, $n \ge 2$.  Let $A$ be an associative commutative algebra with unit.  In this subsection we will construct some finite-dimensional irreducible $\lie p(n) \otimes A$-modules which will be used in the next subsection.

For each $\psi\in (\lie h\otimes A)^*$, denote by $\C_\psi$ the one-dimensional $\lie b \otimes A$-module where the action of $\lie n^+\otimes A$ is trivial and the action of $\lie h\otimes A$ is given by $\psi$.  Define a Verma type module $M(\psi)$ to be
\[
M(\psi) = U(\g\otimes A) \otimes_{U(\lie b \otimes A)} \C_\psi.
\]
Since a submodule of $M(\psi)$ is proper if and only if its intersection with $\C_\psi$ is trivial, $M(\psi)$ admits a unique maximal non-proper submodule. Let $V(\psi)$ denote the unique irreducible quotient of $M(\psi)$ by such a submodule.

\begin{proposition}
Every finite-dimensional irreducible $\ga$-module is isomorphic to $V(\psi)$, for some $\psi \in (\lie h \otimes A)^*$.
\end{proposition}

\begin{proof}
The proof of \cite[Proposition~4.5]{savage14} only requires the existence of a nonzero weight vector $v \in V$. Since $\lie h \otimes A$ is abelian, such a vector always exists.
\end{proof}

The annihilator $\Ann_A (V)$ of a $\ga$-module $V$ is, by definition, the sum of all ideals $I$ of $A$ such that $(\g \otimes I)V=0$. The support of $V$ is defined to be
\[
\Supp (V) = \Supp (\Ann_A(V)).
\]

\begin{proposition}\label{prop:irred.tensor}
The tensor product of two irreducible finite-dimensional $\ga$-modules with disjoint supports is irreducible as well.
\end{proposition}

\begin{proof}
This proof is similar to that of \cite[Proposition~4.12]{savage14}.

\details{
Assume that $V_1$, $V_2$ are irreducible finite-dimensional $\ga$-modules with disjoint supports, and let $\rho_1$ and $\rho_2$ be the representations associated to $V_1$ and $V_2$, respectively. If $I_1$ and $I_2$ denote their annihilators, then the representation $\rho_1\otimes \rho_2$ associated to the action of $\ga$ on $V_1\otimes V_2$ factors through the composition 
\begin{equation} \label{eq:tensor-product-factor-map}
\ga \hookrightarrow
(\ga) \oplus (\ga) \twoheadrightarrow
(\ga/I_1) \oplus (\ga/I_2),
\end{equation}
where the injective homomorphism on the left is the diagonal map and the surjective homomorphism on the right is induced by the quotient maps. Notice that for $i=1,2$, $V_i$ is an irreducible finite-dimensional $\ga/I_i$-module, and its highest weight space has dimension one. Since homomorphisms of modules preserve weight spaces, we have $\End_\ga (V_1) \cong \End_\ga (V_2) \cong \C$. Therefore, it follows from \cite[Proposition~8.4]{Che95} that $V_1\otimes V_2$ is an irreducible $(\ga/I_1) \oplus (\ga/I_2)$-module.

Now, by the fact that the supports of $I_1$ and $I_2$ are disjoints, we have that $I_1\cap I_2 = I_1 I_2$ and then $A=I_1+I_2$. Therefore $\g\otimes (A/I_1I_2)\cong (\ga/I_1)\oplus (\ga/I_2)$, and we have the following commutative diagram:
\[
\xymatrix@R=6ex@C=4ex{
\ga \ar[r]^-{\Delta} \ar@{->>}[d] & (\ga) \oplus (\ga) \ar@{->>}[d] \\
\ga/I_1I_2 \ar[r]^-{\cong} & (\ga/I_1) \oplus (\ga/I_2).
}
\]
This proves that \eqref{eq:tensor-product-factor-map} is surjective, and therefore our result follows.
\qedhere
}
\end{proof}

This proof of the next result follows ideas contained in \cite[Theorem~4.16]{savage14} (see also \cite[Lemma~2.3]{Bag15} and \cite[Proposition~4.3]{BLL15}).

\begin{proposition}\label{prop:quasi-fin.and.ann}
Let $\psi\in (\lie h\otimes A)^*$. The weight spaces of $V(\psi)$ are finite dimensional if and only if there exists an ideal $I$ of $A$ of finite codimension such that $(\g\otimes I)V(\psi)=0$.
\end{proposition}

\begin{proof}
Suppose all the weight spaces of $V(\psi)$ are finite dimensional and let $v$ be a highest weight vector of $V(\psi)$. Let $\Delta^-$ denote the set $\{ -\varepsilon_i \pm \varepsilon_j \mid 1 \le i < j \le n \}$.  For each $\alpha \in \Delta^-$, define $I_\alpha$ to be $\{ a \in A \mid (\g_{\alpha}\otimes a)v=0\}$.  Since each weight space of $V(\psi)$ is finite dimensional, $I_\alpha$ is an ideal of $A$ of finite codimension. Let $I$ be $\bigcap_{\alpha \in \Delta^-} I_\alpha$.  Since $\g$ is finite dimensional, $I$ is an intersection of finitely many finite-codimensional ideals.  In particular, $I$ is also a finite-codimensional ideal of $A$.  We claim that $(\g\otimes I)V(\psi)=0$.  Indeed, the fact that $(\lie n^+\otimes A)v=0$ follows from the fact that $v$ is a highest weight vector.  The fact that $(\lie n^-\otimes I)v=0$ follows from the construction of $I$.  Finally, notice that $\lie h\subseteq [\lie n^-,\lie n^+]$, and so $(\lie h\otimes I)v\subseteq [\lie n^-\otimes I,\lie n^+\otimes A]v=0$.  Thus $(\g\otimes I)v=0$ and hence $W=\{w\in V(\psi) \mid (\g\otimes I)v=0\}$ is a non-trivial submodule of $V(\psi)$.  Since $V(\psi)$ is irreducible, we conclude that $W=V(\psi)$.

Suppose now there exists a finite-codimensional ideal $I$ of $A$ such that $(\g\otimes I)V(\psi)$.  Then the action of $\ga$ on $V(\psi)$ factors through an action of the finite-dimensional Lie superalgebra $\ga/I$.  Thus, by standards arguments using the PBW Theorem, all the weight spaces of $V(\psi)$ are finite dimensional.
\end{proof}

\begin{center}\textit{
For the remainder of this section, we will assume that $A$ is finitely generated.
}\end{center}

\begin{proposition}
Let $\psi\in (\lie h\otimes A)^*$.  The support of $V(\psi)$ is finite if and only if there exists a finite-codimensional ideal $I$ of $A$ such that $(\g\otimes I)V(\psi)=0$.
\end{proposition}

\begin{proof}
Recall that $\Supp V(\psi) = \Supp \Ann_A(V(\psi))$.  So $\Ann_A(V(\psi))$ has finite codimension if and only if there exists a finite-codimensional ideal $I$ of $A$ such that $(\g\otimes I)V(\psi)=0$.  Since $A$ is finitely generated, $\Ann_A(V(\psi))$ has finite codimension if and only if its support is finite.
\end{proof}

Before stating the next result, recall that $\g_{\bar 0}$ is a finite-dimensional simple Lie algebra.

\begin{proposition}\label{prop:radical.ann}
If $V$ is an irreducible finite-dimensional $\ga$-module, then $(\g\otimes J)V=0$ for some radical ideal $J$ of $A$ of finite codimension.
\end{proposition}

\begin{proof}
Since $V$ is irreducible and has finite dimension, Proposition~\ref{prop:quasi-fin.and.ann} implies that $(\g\otimes I)V=0$ for some finite-codimensional ideal $I$ of $A$.  Let $J=\sqrt I$ be the radical of $I$.  We will show that $(\g\otimes J)V=0$.  Since we are assuming that $A$ is finitely generated (and in particular, Noetherian), there exists some power of $J$ that is contained in $I$.  Then $\g\otimes (I/J)$ is a solvable Lie superalgebra satisfying the following property:
\[
\left[ (\g\otimes (J/I))_{\bar 1},(\g\otimes (J/I))_{\bar 1}\right]
= \left[ \g_{\bar 1}, \g_{\bar 1} \right] \otimes (J^2/I)
\subseteq \g_{\bar 0}\otimes (J^2/I)
= \left[ (\g\otimes (J/I))_{\bar 0}, (\g\otimes (J/I))_{\bar 0} \right],
\]
where the last equality follows from the fact that $\g_{\bar 0}$ is a simple Lie algebra. Therefore, it follows from \cite[Proposition~5.2.4]{kac77}, that any irreducible finite-dimensional representation of $\g\otimes J$ is one-dimensional. Then there exists a nonzero vector $w\in V$, and an element $\mu \in (\g\otimes J)^*$, such that $xw=\mu(x)w$ for all $x\in \g\otimes J$. We claim that $\mu=0$. Since $V$ is finite dimensional, for any $z\in \lie n^\pm\otimes J$, there exists $m\geq 0$ such that $z^mw=\mu(z)^m w=0$. In other words, $\mu(\lie n^\pm\otimes J)=0$, and hence $(\lie n^\pm\otimes J)w=0$. Let $\mu'$ denote the restriction of $\mu$ to $\g_{\bar 0}\otimes J$. Since $\g_{\bar 0}$ is a simple Lie algebra, it follows that the kernel of $\mu'$ must be $\g_{\bar 0}\otimes J$. In particular, $\mu'(\lie h \otimes J)=0$, since $\lie h\subseteq \g_{\bar 0}$. We have thus proved that $(\g\otimes J)w=0$. Now the result follows from the irreducibility of $V$ along with the fact that $W=\{v\in V \mid (\g\otimes J)v=0\}$ is a nonzero submodule of $V$.
\end{proof}

\subsection{Classification of finite-dimensional irreducible modules} \label{ss:class.pn}

Recall that we are assuming that $\g = \lie p (n)$, $n \ge 2$, and that $A$ is an associative, commutative, finitely-generated algebra with unit.  In this subsection we will classify all the finite-dimensional irreducible $\Gga$-modules in the case where $\Gamma$ is a finite abelian group acting on $\g$ and $A$ by automorphisms and such that the induced action of $\Gamma$ on $\maxspec (A)$ is free.

Recall from Section~\ref{ss:map.superalgebras} that for every $\sfm \in \maxspec(A)$ one may consider the composition
\[
{\rm ev}_{\sfm} 
\colon \ga 
\to \ga / \lie g \otimes \sfm
\iso \g.
\]
Furthermore for a $\g$-module $V$ with associated representation $\rho \colon \g \to \lie{gl}(V)$, the module $\ev{\sfm} (V)$ is defined to be the $\ga$-module with associated representation given by the pull-back of $\rho$ along ${\rm ev}_{\sfm}$:
\[
\ev{\sfm}(\rho) \colon
\ga \pra[{\rm ev}_{\sfm}]
\g \pra[\ \rho \ ]
\lie{gl} (V).
\]
The $\ga$-module $\ev{\sfm} (V)$ is called an evaluation module and its associated representation $\ev{\sfm}(\rho)$ is called an evaluation representation.  Define $\Gev{\sfm}(\rho)$ to be the restriction of $\ev{\sfm}(\rho)$ to $\Gga$.

\begin{theorem}\label{thm:irred.classification}
Every irreducible finite-dimensional $\g\otimes A$-module  $V$ is isomorphic to a tensor product of evaluation modules.
\end{theorem}

\begin{proof}
Since $V$ is finite dimensional, by Proposition~\ref{prop:radical.ann}, there exists a radical ideal $I$ of $A$ of finite codimension such that $(\g\otimes I)V=0$. Since $A$ is finitely generated and $I$ has finite codimension, Lemma~\ref{lem:assoc-alg-facts}~\eqref{lem-item:finiteCod-and-finiteSupp} implies that the support of $I$ is finite. If $\Supp (I) = \{\sfm_1,\ldots, \sfm_n\} \subseteq \maxspec(A)$, then $I=\sqrt{I}=\sfm_1\cdots \sfm_n$. Therefore the action of $\g\otimes A$ on $V$ factors through the map
\begin{equation} \label{eq:quotient}
\g\otimes A
\twoheadrightarrow \g\otimes A/I
\iso \bigoplus_{i=1}^n \g\otimes A/\sfm_i
\iso \g^{\oplus n}.
\end{equation}
In other words, $V$ is isomorphic to the pull-back of an irreducible finite-dimensional $\g^{\oplus n}$-module $W$ along \eqref{eq:quotient}.  Notice that $W$ is also an irreducible finite-dimensional module for $U\left( \g^{\oplus n}\right) \cong U (\g)^{\otimes n}$.  By \cite[Proposition~8.4]{Che95}, there exist irreducible finite-dimensional modules $V_1, \dotsc, V_n$ for $U(\g)$ such that $W$ is either isomorphic to $V_1 \otimes \dotsm \otimes V_n$ or to a proper submodule of $V_1\otimes \dotsm \otimes V_n$.  Since $V_1 \otimes \dotsm \otimes V_n$ is irreducible by Proposition~\ref{prop:irred.tensor}, $W \cong V_1 \otimes \dotsm \otimes V_n$, and $V$ is isomorphic to a tensor product of evaluation modules.
\end{proof}

\begin{center}\textit{
\begin{minipage}{.82\textwidth}
From now on we will assume that $\Gamma$ is a finite abelian group acting on $\g$ and $A$ by automorphisms and such that the induced action of $\Gamma$ on $\maxspec (A)$ is free.
\end{minipage}
}
\end{center}

Let $\irred(\g)$ (resp. $\irred\Gga$) be the set of isomorphism classes of irreducible finite-di\-men\-sion\-al modules for $\g$ (resp. $\Gga$).  Let $[V] \in \irred (\g)$ denote the isomorphism class of a $\g$-module $V$.  Notice that, if $V$ and $V'$ are isomorphic \g-modules, then $\Gev{\sfm}(V)$ and $\Gev{\sfm} (V')$ are isomorphic \Gga-modules.  Therefore, for each $[V] \in \irred(\g)$, we define $\Gev{\sfm}[V]$ to be the isomorphism class of $\Gev{\sfm}(V)$ in $\irred\Gga$.

Also recall that the action of $\Gamma$ on $\g$ induces an action of $\Gamma$ on $\irred(\g)$.  Namely, if $V$ is a \g-module representative of $[V] \in \irred(\g)$ with associated representation $\rho \colon \g \to \lie{gl} (V)$, then $\gamma [V] = [V^\gamma]$, where $V^\gamma$ is a \g-module with underlying vector space $V$ and associated representation $\rho' \colon \g \to \lie{gl}(V)$ given by $\rho'(x) = \rho (\gamma^{-1} x)$ for all $x \in \g$.

Let $\cal P$ be the set of $\Gamma$-equivariant functions $\pi \colon \maxspec(A) \to \irred(\g)$ such that $\pi(\sfm)=[\C]$ for all but finitely many distinct $\sfm \in \maxspec(A)$.  Given $\pi \in \cal P$, recall that its support is defined to be $\Supp(\pi) = \{ \sfm \in \maxspec(A) \mid \pi(\sfm) \textup{ is nontrivial} \}$.  Let $X_*$ be denote the set of all finite subsets $\sfM \subseteq \maxspec(A)$ satisfying the following property: if $\sfm$ and $\sfm'$ are distinct elements in $\sfM$, then $\sfm\notin\Gamma\sfm'$.  As in Section~\ref{ss:map.superalgebras}, for each $\pi \in \cal P$, fix an element $\Supp_ *(\pi)$ in $X_*$ containing one element of each $\Gamma$-orbit in $\Supp(\pi)$, and define $\cal V(\pi)$ to be the $\Gga$-module
\[
\cal V(\pi) =
\bigotimes_{\sfm \in \Supp_*(\pi)} \Gev{\sfm}\pi(\sfm).
\]

\begin{lemma}\label{lem:properties_map}
With the above notation, the following hold:
\begin{enumerate}[(a)]
\item \label{lem:ev.well.def}
For every $\pi\in \cal P$, the isomorphism class of $\cal V(\pi)$ does not depend on the choice of $\Supp_*(\pi)$.

\item \label{lem:ev.irred}
For every $\pi\in \cal P$, the $\Gga$-module $\cal V(\pi)$ is irreducible.

\item \label{lem:properties_map.injection}
The map $\cal P \to \irred \Gga$ given by $\pi \mapsto \cal V(\pi)$ is injective.
\end{enumerate}
\end{lemma}

\begin{proof}
Part~\eqref{lem:ev.well.def} follows from \cite[Lemma~5.9]{savage14}. Part~\eqref{lem:ev.irred} follows from Proposition~\ref{prop:irred.tensor} along with the fact that the map $ \Gev{\sfm}$ is surjective for all $\sfm \in \maxspec(A)$. Part~\eqref{lem:properties_map.injection} follows from \cite[Proposition~5.11]{savage14}.  Notice that the condition of $\g$ being basic is not used in the proofs of the results cited from \cite{savage14}.
\end{proof}

\begin{proposition}\label{prop:restriction.property}
Every finite-dimensional $\Gga$-module $V$ is isomorphic to the restriction of a finite-di\-men\-sion\-al $\ga$-module $V'$ whose support is in $X_*$.  Moreover, $V$ is irreducible if and only if $V'$ is.
\end{proposition}

\begin{proof}
The proof of this fact for any finite-dimensional simple Lie superalgebra is the same as the proof of \cite[Proposition~8.5]{savage14}.  Notice that, although the fact that $\Supp (V')$ is an element of $X_*$ is not stated, it is also proved there.
\end{proof}

\begin{theorem}\label{thm:main.p(n)}
Let $A$ be an associative, commutative, finitely-generated algebra with unit, $\Gamma$ be a finite abelian group acting on $A$ and $\g$ by automorphisms, and such that the induced action of $\Gamma$ on $\maxspec(A)$ is free.  The map $\cal P \to \irred \Gga$ given by $\pi \mapsto \cal V(\pi)$ is a bijection.
\end{theorem}

\begin{proof}
Recall from Lemma~\ref{lem:properties_map}~\eqref{lem:properties_map.injection} that the map $\cal P \to \irred \Gga$ given by $\pi \mapsto \cal V(\pi)$ is injective.  Let $V$ be an irreducible finite-dimensional $\Gga$-module.  By Proposition~\ref{prop:restriction.property}, $V$ is isomorphic to the restriction of an irreducible finite-dimensional $\ga$-module $V'$, whose support is in $X_*$.  Hence, by Theorem~\ref{thm:irred.classification}, $V' \cong \bigotimes_{i=1}^n \ev{\sfm_i}(V_i)$ for some $n \ge 0$, $\{ \sfm_1,\dotsc, \sfm_n \}\in X_*$ and irreducible finite-dimensional \g-modules $V_1, \dotsc, V_n$.  Thus, $V$ is isomorphic to $\cal V (\pi)$, where $\pi (\sfm_i)=[V_i]$ for all $i \in \{1, \dotsc, n\}$, and $\pi (\sfm)=[\C]$ for all $\sfm \not\in \Supp(V')$.
\end{proof}

\section{Transgression maps} \label{s:transgression}

Let $\G$ be a Lie superalgebra, and consider its exterior algebra $\Lambda^\bullet \G$ with the $\G$-module structure induced from the adjoint representation.  Given a proper ideal $\I \subset \G$, define an increasing filtration $0 = \Lambda^\bullet_0 \subsetneq \Lambda^\bullet_1 \subsetneq \Lambda^\bullet_2 \subsetneq \dots \subsetneq \Lambda^\bullet = \Lambda^\bullet \G$ by:
\begin{gather*}
\Lambda^n_p
:= \cspan \{ x_1 \wedge \dots \wedge x_n \mid x_{i_1}, \dotsc, x_{i_k} \in {\I}\textup{ for some } k > n-p, \ 1 \le i_1 < \dotsb < i_k \le n \},
\ 0 < p \le n, \\
\Lambda^n_p 
:= \Lambda^n \G \quad \textup{for all } p > n.
\end{gather*}
Now, given a $\G$-module $M$, define a decreasing filtration $\dots \subsetneq C^\bullet_2 \subsetneq C^\bullet_1 \subsetneq C^\bullet_0 = \hom_\C \left( \Lambda^\bullet \G, M \right)$ by
\[
C^\bullet_p
:= \left\{ f \colon \Lambda^\bullet \to M \mid f(\Lambda^\bullet_p) = 0 \right\}.
\]

Recall that $\h^\bullet (\G, M)$, the cohomology of $\G$ with coefficients in $M$, is the cohomology of the cocomplex $\left( \hom_\C (\Lambda^\bullet \G, M), \partial^\bullet \right)$, where $\partial^\bullet$ is given in \eqref{eq:codiffs}.  Moreover, notice that $\partial^n (C^n_p) \subseteq C^{n+1}_p$ for all $p, n \ge 0$.

In this section, we will review the construction of the Lyndon-Hochschild-Serre spectral sequence that converges to $\h^\bullet (\G, M)$ in the case where $\I \cdot M = 0$.  Even though most of this material is not new (see \cite{HS53} and \cite[\S16.6]{Mus12}), we will use it to describe the transgression map in a few cases that will be important in the proof of Theorem~\ref{thm:main} (see Proposition~\ref{prop:ker.transgression}).

\subsection{$E_0$-page} \label{ss:E0.lhs}

Let $E_0^{p,q} := C^{p+q}_p / C^{p+q}_{p+1}$ for all $p, q \ge 0$.  Recall that $C^n_p$ consists of linear maps in $\hom_\C \left( \Lambda^n, M \right)$ that vanish on $\Lambda^n_p$ (and similarly for $C^n_{p+1}$).  Hence, we have the following isomorphisms of vector spaces
\begin{equation} \label{eq:E0.iso}
E_0^{p,q}
\cong \frac{\hom_\C (\Lambda^{p+q} / \Lambda^{p+q}_{p}, M)}{\hom_\C (\Lambda^{p+q} / \Lambda^{p+q}_{p+1}, M)}
\cong \hom_\C (\Lambda^{p+q}_{p+1} / \Lambda^{p+q}_p, M)
\qquad \textup{ for all } p, q \ge 0,
\end{equation}
where the second isomorphism is given by the First Isomorphism Theorem and the restriction of linear maps in $\hom_\C (\Lambda^{p+q} / \Lambda^{p+q}_{p}, M)$ to $\Lambda^{p+q}_{p+1}/\Lambda^{p+q}_p$.

Now recall that, as vector spaces, $\G \cong {\I}\oplus (\G/{\I})$ and thus, $\Lambda^n \G \cong \bigoplus_{j=0}^n \Lambda^{n-j}(\I) \otimes \Lambda^j (\G/{\I})$.  Using these isomorphisms and the definition of $\Lambda^n_p$, we see that, as vector spaces,
\begin{equation} \label{eq:Lambda.iso}
\Lambda^n_p \cong \bigoplus_{j=0}^{p-1} \Lambda^{n-j}(\I) \otimes \Lambda^j (\G/{\I})
\quad \textup{ and } \quad
\Lambda^{p+q}_{p+1} / \Lambda^{p+q}_p \cong \Lambda^q(\I) \otimes \Lambda^p (\G/{\I})
\quad \textup{ for all } \ p, q, n \ge 0.
\end{equation}
Coupling these isomorphisms with the isomorphisms \eqref{eq:E0.iso}, we obtain the following isomorphisms of vector spaces:
\[
E_0^{p,q} \cong \hom_\C (\Lambda^q(\I) \otimes \Lambda^p (\G/{\I}), M)
\quad \textup{ for all } \ p, q \ge 0.
\]
Explicitly, a linear map $f \in \hom_\C(\Lambda^q(\I) \otimes \Lambda^p (\G/{\I}), M)$ corresponds to the restriction of a linear map $F \in \hom_\C(\Lambda^{p+q}, M)$ (that vanishes on $\Lambda^{p+q}_{p}$) to $\Lambda^{p+q}_{p+1}$, and the elements in $\G/\I$ are identified with elements in $\G$ via a vector space splitting of the canonical projection $\G \twoheadrightarrow \G/\I$.

Consider the map $D_0^{p,q} \colon E_0^{p,q} \to E_0^{p, q+1}$, explicitly given by $D_0^{p,q}(F + C^{p+q}_{p+1}) = \partial^{p+q}F + C^{p+q+1}_{p+1}$.  Using the explicit formula of $\partial^{p+q}F$ \eqref{eq:codiffs} and the isomorphisms \eqref{eq:E0.iso}, \eqref{eq:Lambda.iso}, $D_0^{p,q}$ induces a map $d_0^{p,q} : \hom_\C (\Lambda^q(\I) \otimes \Lambda^p (\G/{\I}), M) \to \hom_\C (\Lambda^{q+1}(\I) \otimes \Lambda^p (\G/{\I}), M)$, which can be explicitly written as
\begin{align} \label{d0.lhs}
d_0^{p,q}f 
{}&{}(i_0 \wedge \dotsb \wedge i_q \otimes x_1 \wedge \dotsb \wedge x_p) \notag \\
{}&{}= \sum_{0 \le \ell \le q} (-1)^{\ell+|i_\ell|(|F|+|i_0| + \dotsb + |i_{\ell-1}|)} i_\ell \, F(i_0 \wedge \dotsb \wedge \widehat{i_\ell} \wedge \dotsb \wedge i_q \wedge {\widetilde x}_1 \wedge \dotsb \wedge {\widetilde x}_p) \\
&{\quad}+ \sum_{0 \le j < k \le q} (-1)^{\sigma(j,k)} F(i_0 \wedge \dotsb i_{j-1} \wedge [i_j, i_k] \wedge i_{j+1} \wedge \dotsb \wedge \widehat{i_k} \wedge \dotsb \wedge i_q \wedge {\widetilde x}_1 \wedge \dotsb \wedge {\widetilde x}_p),  \notag
\end{align}
where $\sigma(j,k) = k+|i_k|(|i_{j+1}|+ \dotsb + |i_{k-1}|)$, $i_0, \dotsc, i_q \in {\I}$, $x_1, \dotsc, x_p \in \G/\I$, and ${\widetilde x}_1, \dotsc, {\widetilde x}_p$ are their corresponding representatives in $\G$.
\details{
Notice that \eqref{d0.lhs} does not, in fact, depend on the choice of these representatives, as any other choice would differ from this one by extra terms in ${\I}$, and $f$ vanishes in $\Lambda^{p+q}_{p+1}$, that is, when more than $q$ of its arguments are in ${\I}$.
}

Now recall that, for each $p,q \ge 0$, the tensor-hom adjunction induces an isomorphism of vector spaces $E_0^{p,q} \cong \hom_\C \left( \Lambda^q (\I), \hom_\C (\Lambda^p (\G/{\I}), M) \right)$.  Using this isomorphism and \eqref{d0.lhs}, one can check that, for each $p, q \ge 0$, the map $D_0^{p,q}$ corresponds to the $q$-th differential of the cocomplex $\left( \hom_\C \left( \Lambda^\bullet (\I), \hom_\C (\Lambda^p (\G/{\I}), M) \right),\, \partial^\bullet_{\I}\right)$, used to compute $\h^\bullet (\I, \hom_\C (\Lambda^p (\G/{\I}), M))$.

\details{
In fact, $\hom_\C (\Lambda^p (\G/{\I}), M)$ admits a $\G$-module structure explicitly given by $(sf)(x) = s(f(x)) -(-1)^{|s||f|} f(s \cdot x)$ for all homogeneous $s \in \G$, $f \in \hom_\C (\Lambda^p (\G/{\I}), M)$, and $x \in \Lambda^p (\G/{\I})$.  Since $[{\I},\G]\subseteq {\I}$, it follows that $i \cdot x = 0$ for all $i \in \I$ and $x \in \Lambda^p (\G/{\I})$.  Hence, the $\I$-module structure on $\hom_\C (\Lambda^p (\G/{\I}), M)$ is given by $(if) (x) = i (f(x))$ for all $i \in {\I}$, $x \in \Lambda^p (\G/{\I})$ and $f \in \hom_\C (\Lambda^p (\G/{\I}), M)$.  The claim follows from the tensor-hom adjunction, \eqref{eq:codiffs} and \eqref{d0.lhs}.
}

\subsection{$E_1$-page} \label{ss:E1.lhs}

Let $E_1^{p, q} := \h^q  \left( E_0^{p, \bullet}, D_0^{p, \bullet} \right)$ for all $p, q \ge 0$.  As a direct consequence of the arguments of Section~\ref{ss:E0.lhs}, there are isomorphisms of vector spaces:
\[
E_1^{p,q} \cong \h^q ({\I}, \hom_\C (\Lambda^p (\G/{\I}), M))
\quad \textup{ for all } \ p,q \ge 0.
\]

Moreover, for each $p, q \ge 0$, there is an isomorphism of vector spaces
\[
\h^q ({\I}, \hom_\C (\Lambda^p (\G/{\I}), M))
\cong \hom_\C (\Lambda^p (\G/{\I}), \h^q ({\I}, M)).
\]
\details{
By definition, every element in $\h^q ({\I}, \hom_\C (\Lambda^p (\G/{\I}), M))$ is an equivalence class of linear maps in $\hom_\C (\Lambda^q(\I), \hom_\C(\Lambda^p(\G/\I), M))$.  By the tensor-hom adjunction, there are isomorphisms $\hom_\C (\Lambda^q(\I), \hom_\C (\Lambda^p (\G/{\I}), M)) \cong \hom_\C (\Lambda^q(\I) \otimes \Lambda^p (\G/{\I}), M) \cong \hom_\C (\Lambda^p (\G/{\I}), \hom_\C(\Lambda^q(\I), M))$.  Under these isomorphisms, the differential $\partial^q_{\I}$ of the cocomplex $\hom_\C (\Lambda^\bullet(\I), \hom_\C(\Lambda^p(\G/\I), M))$ corresponds to the differential $d_0^{p,q} : \hom_\C (\Lambda^q(\I) \otimes \Lambda^p (\G/{\I}), M) \to \hom_\C (\Lambda^{q+1}(\I) \otimes \Lambda^p (\G/{\I}), M)$, which corresponds, on the cocomplex $\hom_\C(\Lambda^p(\G/\I), \hom_\C(\Lambda^\bullet(\I), M))$, to the post-composition with the differential $\partial^q_{\I, M}$.  This implies that $\h^q(\I, \hom_\C(\Lambda^p(\G/\I), M))$ is isomorphic to $E_1^{p,q}$, which is isomorphic to $\hom_\C (\Lambda^p(\G/\I), \h^q(\I, M))$.

}
Now, recall that $\h^q ({\I}, M)$ is a $\G/{\I}$-module (that is, a $\G$-module with trivial ${\I}$-action) with its $\G$-module structure induced from $\hom_\C (\Lambda^q(\I), M)$.
\details{
Explicitly, $(sf) (x) = s(f(x)) -(-1)^{|s||f|} f(s \cdot x)$ for all homogeneous $s \in \G$, $f \in \hom_\C (\Lambda^q(\I), M)$, and $x \in \Lambda^q(\I)$.  The fact that $\I$ acts trivially on $\h^q(\I, M)$ follows from the fact that the elements in $\h^q (\I, M)$ are equivalence classes (modulo $\im \partial^{q-1}_{\I, M}$) of linear maps $h \in \hom_\C (\Lambda^q(\I), M)$ such that $\partial^q_{\I, M}(h) = 0$.
}
Thus, one can consider $\h^\bullet (\G/\I, \h^q(\I, M))$, the cohomology of the cocomplex $\hom_\C (\Lambda^\bullet (\G/{\I}), \h^q({\I}, M))$, whose differential will be denoted by $\partial_{\G/{\I}}^\bullet$ and is given by \eqref{eq:codiffs}.

Since every element in $E_1^{p,q}$ is an equivalence class (modulo $\im D_0^{p, q-1}$) of elements in $E_0^{p,q}$, and every element in $E_0^{p,q}$ is the restriction (to $\Lambda^{p+q}_{p+1}$) of a linear map in $\hom_\C(\Lambda^{p+q}, M)$, we can define a map $D_1^{p,q} \colon E_1^{p,q} \to E_1^{p+1, q}$ by $D_1^{p,q}(F + \im D_0^{p, q-1}) = \partial^{p+q}F + \im D_0^{p+1, q-1}$.  Using the explicit formula for $\partial^{p+q}F$ \eqref{eq:codiffs}, the tensor-hom adjunction and \eqref{d0.lhs}, we obtain that
\begin{equation} \label{d1.lhs}
\partial^{p+q}F ({\widetilde x}_0 \wedge \dotsb \wedge {\widetilde x}_p \wedge i_1 \wedge \dotsb \wedge i_q)
= (\partial^p_{\G/{\I}}F ({\widetilde x}_0 \wedge \dots \wedge {\widetilde x}_p)) (i_1 \wedge \dots \wedge i_q) + \im D_0^{p+1,q-1},
\end{equation}
for all homogeneous $i_1, \dotsc, i_{q-1} \in {\I}$, $x_0, \dotsc, x_{p+1} \in \G / {\I}$, and their corresponding representatives ${\widetilde x}_0, \dotsc, {\widetilde x}_{p+1} \in \G$.  This implies that $D_1^{p,q} : E_1^{p,q} \to E_1^{p+1,q}$ induces a map
\[
d_1^{p,q} : \hom_\C (\Lambda^p(\G/\I), \h^q(\I, M)) \to \hom_\C (\Lambda^{p+1}(\G/\I),  \h^q(\I, M)),
\]
which is the $p$-th differential of the cocomplex $( \hom_\C \left( \Lambda^\bullet (\G/{\I}), \h^q({\I}, M) \right),\, \partial_{\G/{\I}}^\bullet )$.

\subsection{$E_2$-page and the transgression map} \label{ss:E2+trans}

Let $E_2^{p,q} := \h^p \left( E_1^{\bullet, q}, D_1^{\bullet, q} \right)$ for all $p, q \ge 0$.  As a direct consequence of the arguments of Section~\ref{ss:E1.lhs}, there are isomorphisms
\begin{equation} \label{eq:E2.iso}
E_2^{p,q} \cong \h^p \left( \G/{\I}, \h^q ({\I}, M) \right)
\quad \textup{ for all } \ p, q \ge 0.
\end{equation}
It is known (see, for instance, \cite[Chapter 1, \textsection 6.5]{fuks86}) that the spectral sequence thus obtained converges to the cohomology of $\G$ with coefficients in $M$; that is,
\begin{equation} \label{eq:LHSss}
E_2^{p,q}
\cong \h^p \left( \G/{\I}, \h^q ({\I}, M) \right)
\Rightarrow \h^{p+q} (\G, M).
\end{equation}

Since every element in $E_2^{p,q}$ is an equivalence class (modulo $\im D_1^{p-1, q}$) of elements in $E_1^{p,q}$, every element in $E_1^{p,q}$ is an equivalence class (modulo $\im D_0^{p, q-1}$) of elements in $E_0^{p,q}$, and every element in $E_0^{p,q}$ is the restriction (to $\Lambda^{p+q}_{p+1}$) of a linear map in $\hom_\C(\Lambda^{p+q}, M)$, we can define a map $D_2^{p,q} \colon E_2^{p,q} \to E_2^{p+2, q-1}$ by $D_2^{p,q} (F + \im D_1^{p-1,q}) = \partial^{p+q}F + \im D_1^{p+1, q-1}$.  Using the explicit formula for $\partial^{p+q}F$ \eqref{eq:codiffs} and \eqref{d1.lhs}, we can obtain that
\begin{align} \label{d2.lhs}
&\partial^{p+q}F({\widetilde x}_0 \wedge \dotsb \wedge {\widetilde x}_{p+1} \wedge i_1 \wedge \dotsb \wedge i_{q-1}) \\
&{\ }= \sum_{0 \le j < k \le p+1} (-1)^{\sigma(j,k)} F({\widetilde x}_0 \wedge \dotsb \wedge {\widetilde x}_{j-1} \wedge [{\widetilde x}_j, {\widetilde x}_k] \wedge {\widetilde x}_{j+1} \wedge \dotsb \wedge \widehat{{\widetilde x}_k} \wedge \dotsb \wedge {\widetilde x}_{p+1} \wedge i_1 \wedge \dotsb \wedge i_{q-1}) \notag \\
&{\qquad}+ \im D_1^{p+1, q-1}, \notag
\end{align}
where $\sigma(j,k) = k+|x_k|(|x_{j+1}|+ \cdots + |x_{k-1}|)$, for all homogeneous $i_1, \dotsc, i_{q-1} \in {\I}$, $x_0, \dotsc, x_{p+1} \in \G/\I$, and their corresponding representatives ${\widetilde x}_0, \dotsc, {\widetilde x}_{p+1} \in \G$.  For each $p,q \ge 0$, denote by
\[
d_2^{p,q} : \h^p(\G/\I, \h^q(\I, M)) \to \h^{p+2} (\G/\I, \h^{q-1}(\I, M))
\]
the map induced by $D_2^{p,q}$ via the isomorphism \eqref{eq:E2.iso}.

\begin{definition} \label{defn:transgression}
The transgression map of the Lyndon-Hochschild-Serre spectral sequence \eqref{eq:LHSss} is defined to be
\[
d_2^{0,1} \colon \h^0 (\G/\I, \h^1(\I, M)) \to \h^2(\G/\I, \h^0(\I, M)).
\]
\end{definition}

In the next result, we describe the kernel of this transgression map in two particular cases that are important to the proof of Theorem~\ref{thm:main}.

\begin{proposition} \label{prop:ker.transgression}
Let $\G$ be a Lie superalgebra and $\I \subsetneq \G$ be an ideal.
\begin{enumerate}[(a)]
\item \label{item:ker.transgression.ev.mod}
If $\G/{\I}$ is a Lie subalgebra of $\G$, that is, $[{\widetilde x}_j, {\widetilde x}_k] \in \G \setminus {\I}$ for all ${\widetilde x}_j, {\widetilde x}_k \in \G \setminus {\I}$, then $d_2^{0,1} = 0$.  In particular, if $M$ is a finite-dimensional $\g$-module, $\G = \Gga$ and $\I = \g \otimes I$, $I = \prod_{\gamma \in \Gamma} \gamma\sfm$, $\sfm \in \maxspec(A)$, then the kernel of the transgression map
\[
d_2^{0,1} : \hom_{(\lie g \otimes A/I)^\Gamma} \left( (\lie g \otimes I / I^2)^\Gamma, \Gev{\sfm} M \right)
\to \h^2 \left( (\ga/I)^\Gamma, \Gev{\sfm} M \right)
\]
is isomorphic to $\hom_\g \left( \g, M \right)^{\oplus \dim_{A/\sfm} \sfm / \sfm^2}$.

\item \label{item:ker.transgression.gen.ev.mod}
If $[\G/{\I}, \G/{\I}] = {\I}$, then $d_2^{0,1} f ({\widetilde x}_0 \wedge {\widetilde x}_1) = 0$ for all ${\widetilde x}_0, {\widetilde x}_1$, if and only if $f = 0$.  In particular, if $\G = \Gga$ and $\I = \g \otimes I$, $I = \prod_{\gamma \in \Gamma} (\gamma\sfm)^n$, $\sfm \in \maxspec(A)$, $n > 1$, then the kernel of the transgression map $d_2^{0,1} : \hom_{(\lie g \otimes A/I)^\Gamma} \left( (\lie g \otimes I / I^2)^\Gamma, \Gev{\sfm^n} M \right)
\to \h^2 \left( (\ga/I)^\Gamma, \Gev{\sfm^n} M \right)$ is $0$.
\end{enumerate}
\end{proposition}

\begin{proof}
Part~\eqref{item:ker.transgression.ev.mod} follows from \eqref{d2.lhs}, the fact that $(\ga/I)^\Gamma \cong \g$, and the fact that $(\g \otimes I/I^2)^\Gamma \cong \g \otimes \sfm / \sfm^2 \cong \g^{\oplus \dim_{A/\sfm} \sfm / \sfm^2}$. To prove part~\eqref{item:ker.transgression.gen.ev.mod}, first notice that, by \eqref{d2.lhs}, we have that $\partial^1 f ({\widetilde x}_0 \wedge {\widetilde x}_1) = (-1)^{1+|x_1| |x_0|} f([{\widetilde x}_0, {\widetilde x}_1])$, for all representatives ${\widetilde x}_0, {\widetilde x}_1 \in \G$ of corresponding elements $x_0, x_1 \in \G/{\I}$. Since $[\G/{\I}, \G/{\I}] = {\I}$ by hypothesis, it follows that $\partial^1 f ({\widetilde x}_0 \wedge {\widetilde x}_1) = 0$ for all ${\widetilde x}_0, {\widetilde x}_1$ if and only if $f(\I) = 0$.
\end{proof}

\section{Extensions} \label{Exts}

Throughout this section, we will assume that $\lie g$ is a finite-dimensional simple Lie superalgebra, that $A$ is an associative, commutative, finitely-generated algebra with unit, that $\Gamma$ is a finite abelian group acting on $\g$ and $A$ by automorphisms, and that the action of $\Gamma$ on $\maxspec (A)$ is free.

We begin by using Lemma \ref{lem:isos} to reduce the problem of describing $p$-extensions between finite-dimensional irreducible $\Gga$-modules to that of describing the $p$-th cohomology of $\Gga$ with coefficients in indecomposable modules.

\begin{proposition} \label{prop:ext=h}
Let $\pi, \pi' \in \cal P$.  There exist $n, \ell, q_1, \dotsc, q_\ell \in \bb Z_{>0}$, maximal ideals $\sfm_1, \dotsc, \sfm_\ell \subseteq A$ in distinct $\Gamma$-orbits, and, for each $i \in \{ 1 , \dotsc, \ell \}$, $j \in \{ 1, \dotsc, q_i \}$, a finite-dimensional indecomposable $\ga / \sfm_i^n$-module $M_{i, j}$, such that 
\begin{gather*}
\hom_\C ( \cal V(\pi), \cal V(\pi') )^{\oplus 2^{\kappa(\pi)+\kappa(\pi')}}
\cong \bigoplus_{\atop{1 \le j_i \le q_i}{1 \le i \le \ell}} \Gev{\sfm_1^n} M_{1, j_1} \otimes \dotsb \otimes \Gev{\sfm_\ell^n} M_{\ell, j_\ell}
\quad \textup{and} \\
\Ext^p_\Gga ( \cal V(\pi), \cal V(\pi') )^{\oplus 2^{\kappa(\pi)+\kappa(\pi')}}
\cong \bigoplus_{\atop{1 \le j_i \le q_i}{1 \le i \le \ell}} \h^p \left( \Gga, \, \Gev{\sfm_1^n} M_{1, j_1} \otimes \dotsb \otimes \Gev{\sfm_\ell^n} M_{\ell, j_\ell} \right),
\quad p>0.
\end{gather*}
\end{proposition}

\begin{proof}
From Remark~\ref{rmk:gather.gen.ev}, there exist $n, \ell \in \bb Z_{>0}$, maximal ideals $\sfm_1, \dotsc, \sfm_\ell \subseteq A$ in distinct $\Gamma$-orbits, and finite-dimensional irreducible $\ga/\sfm_i^n$-modules $V_i$ and $V'_i$ such that
\[
\cal V (\pi) \cong \bightimes_{i=1}^\ell \Gev{\sfm_i^n} V_i
\quad \textup{and} \quad
\cal V (\pi') \cong \bightimes_{i=1}^\ell \Gev{\sfm_i^n} V'_i.
\]
Thus, there are isomorphisms of $\Gga$-modules
\begin{align*}
\hom_\C \left( \cal V (\pi), \cal V (\pi') \right)^{\oplus 2^{\kappa(\pi)+\kappa(\pi')}}
&\cong \hom_\C \left( \cal V (\pi)^{\oplus 2^{\kappa(\pi)}}, \cal V (\pi')^{\oplus 2^{\kappa(\pi')}} \right) \\
&\cong \left( \bigotimes_{i=1}^\ell \Gev{\sfm_i^n} V_i^* \right)
\otimes \left( \bigotimes_{i=1}^\ell \Gev{\sfm_i^n} V'_i \right) \\
&\cong \bigotimes_{i=1}^\ell \Gev{\sfm_i^n} \left( V_i^* \otimes V'_i \right).
\end{align*}
For each $i \in \{1, \dotsc, \ell\}$, since $V^*_i$ and $V'_i$ are finite dimensional, there exist $q_i > 0$ and indecomposable $\lie g\otimes A/\sfm_i^n$-modules, $M_{i,1}, \dotsc, M_{i, q_i}$, such that $V^*_i \otimes V'_i \cong \bigoplus_{j=1}^{q_i} M_{i,j}$. Thus there exist isomorphisms of $\Gga$-modules
\[
\bigotimes_{i=1}^\ell \Gev{\sfm_i^n} \left( V_i^* \otimes V'_i \right)
\cong \bigotimes_{i = 1}^\ell \left( \bigoplus_{j_i=1}^{q_i} \Gev{\sfm_i^n} \left( M_{i, j_i} \right) \right)
\cong \bigoplus_{\atop{1 \le j_i \le q_i}{1 \le i \le \ell}}
\Gev{\sfm_1^n} M_{1, j_1} \otimes \dotsb \otimes \Gev{\sfm_\ell^n} M_{\ell, j_\ell}.
\]
This proves the first statement.  The second statement follows from the first one and Lemma~\ref{lem:isos}.
\end{proof}

This next result is a particular case of Proposition~\ref{prop:ext=h} that will be used in to prove Corollary~\ref{cor:Ext1=0}.

\begin{corollary} \label{cor:ext=h.disjoint.supp}
Let $V$ and $V'$ be finite-dimensional irreducible \Gga-modules.  If the supports of $V$ and $V'$ are disjoint, then $V^* \htimes V'$ is irreducible and, for all $p>0$,
\[
\Ext_\Gga^p (V, V')
\cong
\begin{cases}
\h^p (\Gga, V^* \htimes V'), & \textup{ if } V^*\otimes V'\textup{ is irreducible},\\
\h^p (\Gga, V^* \htimes V')^{\oplus 2}, & \textup{ otherwise}.
\end{cases}
\]
\end{corollary}

Our next goal is to reduce the problem of determining 1-extensions between finite-dimensional $\Gga$-modules to that of determining homomorphisms and extensions between finite-dimensional $\ga/\sfm^n$-modules, where $\sfm$ is a maximal ideal of $A$ and $n$ is a positive integer.  We start with a general result regarding first cohomology.   Recall the definition of transgression map (Definition~\ref{defn:transgression}).

\begin{lemma} \label{lem:h1(ga,M)}
If $M$ is a finite-dimensional $\Gga$-module, then there exists a finite-codi\-men\-sion\-al $\Gamma$-invariant ideal $I \subseteq A$, such that
\[
\h^1 \left( \Gga, M \right)
\cong \h^1 \left( (\lie g \otimes A / I)^\Gamma, M \right) \oplus K,
\]
where $K$ is the kernel of the transgression map 
\[
t_M\colon \hom_{(\lie g \otimes A/I)^\Gamma} \left( (\lie g \otimes I / I^2)^\Gamma, M \right)
\to \h^2 \left( (\ga/I)^\Gamma, M \right).
\]
\end{lemma}

\begin{proof}
Since $M$ is a finite-dimensional $\Gga$-module, by Proposition~\ref{prop:ann.fin.dim}, there exists a $\Gamma$-invariant finite-codimensional ideal $I \subseteq A$ such that $(\g \otimes I)^\Gamma M = 0$.  By Proposition~\ref{prop:lhsss}, there exists a first-quadrant cohomology spectral sequence associated to $\Gga$ and $(\lie g \otimes I)^\Gamma$, namely
\begin{equation} \label{eq:LHSss1}
E_2^{p,q}
\cong \h^p \left( (\lie g \otimes A/I)^\Gamma, \h^q \left( (\lie g \otimes I)^\Gamma, M \right) \right)
\Rightarrow \h^{p+q} \left( \Gga, M \right).
\end{equation}
Since \eqref{eq:LHSss1} is a first-quadrant cohomology spectral sequence, we have an isomorphism of vector spaces $\h^1 (\Gga, M) \cong E_\infty^{1,0} \oplus E_\infty^{0,1}$.  Moreover, 
\[
E_\infty^{1,0}
= E_2^{1,0}
\cong \h^1 \left( (\lie g \otimes A/I)^\Gamma,  M \right)
\quad \textup{ and } \quad
E_\infty^{0,1}
= E_3^{0,1}
\cong \ker (d_2^{0,1} \colon E_2^{0,1} \to E_2^{2,0}),
\]
where $d_2^{0,1}$ is the transgression map $t_M$.

In order to finish the proof, we only need to describe $E_2^{0,1}$ and $E_2^{2,0}$.  By \eqref{eq:LHSss1},
\[
E_2^{0,1} \cong \h^0 \left( (\lie g \otimes A/I)^\Gamma, \h^1 \left( (\lie g \otimes I)^\Gamma, M \right) \right)
\quad \textup{and} \quad
E_2^{2,0} \cong \h^2 ((\lie g \otimes A/I)^\Gamma,  M).
\]
Since $(\lie g \otimes I)^\Gamma$ acts trivially on $M$, by Lemma~\ref{lem:triv.mods}, there is an isomorphism of $(\lie g \otimes A/I)^\Gamma$-modules $\h^\bullet \left( (\lie g \otimes I)^\Gamma, M \right) \cong \h^\bullet \left( (\lie g \otimes I)^\Gamma, \C \right) \otimes M$.  Moreover, by Lemma~\ref{lem:com.even}, $\h^1 \left( (\lie g \otimes I)^\Gamma, \C \right)$ is isomorphic to $((\lie g \otimes I/I^2)^\Gamma)^*$ as a $(\lie g \otimes A/I)^\Gamma$-module.  Thus $E_2^{0,1} \cong \hom_{(\lie g \otimes A/I)^\Gamma} \left( (\lie g \otimes I/I^2)^\Gamma, M \right)$.
\end{proof}

As a consequence of Lemmas~\ref{lem:isos} and \ref{lem:h1(ga,M)}, we obtain the following result.

\begin{corollary} \label{cor:extp.fd}
If $V, V'$ are finite-dimensional irreducible $\Gga$-modules, then $\Ext^1_\Gga (V, V')$ is finite dimensional.
\end{corollary}

\begin{proof}
Since $V$ and $V'$ are finite-dimensional modules, by Lemma~\ref{lem:isos}, $\Ext^1_{\Gga} (V, V')$ is isomorphic to $\h^1 \left( \Gga, V^* \otimes V' \right)$.  By Lemma~\ref{lem:h1(ga,M)}, $\h^1 \left( \Gga, V^* \otimes V' \right)$ is isomorphic to a subspace of
\begin{equation} \label{eq:h1.sub}
\h^1 \left( (\lie g \otimes A / I)^\Gamma, V^* \otimes V' \right) \oplus \hom_{(\lie g \otimes A/I)^\Gamma} \left( (\lie g \otimes I/I^2)^\Gamma, V^* \otimes V' \right),
\end{equation}
where $I$ is a finite-codimensional $\Gamma$-invariant ideal of $A$.  Since $\g$, $A/I$, $I/I^2$, and $V^* \otimes V'$ are finite dimensional, both terms in \eqref{eq:h1.sub} are finite dimensional.  This proves that $\h^1 \left( \Gga, V^* \otimes V' \right)$ is finite dimensional, and finishes the proof.
\end{proof}

We emphasize the relevance of Corollary~\ref{cor:extp.fd} by contrasting it with a case where $\Ext^1$ is not finite dimensional.  Namely, let $\Gamma$ be  a group acting by automorphisms on an abelian Lie superalgebra $\lie a$ and on an associative commutative algebra with unit $B$.  For any finite-dimensional trivial $(\lie a \otimes B)^\Gamma$-modules $M$ and $M'$,
\[
\Ext^1_{(\lie a \otimes B)^\Gamma} (M, M')
\cong \hom_\C \left( (\lie a \otimes B)^\Gamma, M^* \otimes M' \right)
\]
is finite dimensional if and only if $(\lie a \otimes B)^\Gamma$ is finite dimensional.  In particular, when $\Gamma$ is trivial and $B$ is infinite dimensional, $\Ext^1_{(\lie a \otimes B)^\Gamma} (M, M')$ is infinite dimensional.

Recall from Corollary~\ref{cor:h1(ga,C)=0} that $\h^1 (\Gga, \C) = 0$.  The next result gives a vanishing condition for $\h^1 (\Gga, M)$ when $M$ is a finite-dimensional $\Gga$-module of the form $\bigotimes_{i=1}^\ell \Gev{\sfm_i^{n_i}} M_i$, generalizes \cite[Theorem 3.6]{kodera10} and \cite[Theorem 3.7]{NS15}.

\begin{proposition} \label{prop:h1(ga,M)}
Let $\ell, n_1, \dotsc, n_\ell \in \bb Z_{>0}$, $\sfm_1, \dotsc, \sfm_\ell \subseteq A$ be maximal ideals in distinct $\Gamma$-orbits, for each $i \in \{ 1, \dotsc, \ell \}$, let $M_i$ be a finite-dimensional $\lie g \otimes A/\sfm_i^{n_i}$-module, and $M = \bigotimes_{i=1}^\ell \Gev{\sfm_i^{n_i}} M_i$.

\begin{enumerate}[(a)]
\item \label{prop:h1(ga,M)a}
If $\hom_{\lie g \otimes A/\sfm_i^{n_i}} (\C, M_i) = 0$ for more than one index $i$, then $\h^1 \left( \Gga, M \right) = 0$.

\item \label{prop:h1(ga,M)b}
If $\hom_{\lie g \otimes A/\sfm_i^{n_i}} (\C, M_i) = 0$ for exactly one index $i$, then
\[
\h^1 \left( \Gga, M \right)
\cong \h^1 \left( \Gga, \Gev{\sfm_i^{n_i}} M_i \right) \otimes \bigotimes_{j \ne i} \hom_{\lie g \otimes A/\sfm_j^{n_j}} (\C, M_j).
\]

\item \label{prop:h1(ga,M)c}
If $\hom_{\lie g \otimes A/\sfm_i^{n_i}} (\C, M_i) \ne 0$ for all indices $i$, then
\[
\h^1 \left( \Gga, M \right)
\cong \bigoplus_{i=1}^\ell \left( \h^1 \left( \Gga, \Gev{\sfm_i^{n_i}} M_i \right) \otimes \bigotimes_{j \ne i} \hom_{\lie g \otimes A/\sfm_j^{n_j}} (\C, M_j) \right).
\]
\end{enumerate}
\end{proposition}

\begin{proof}
By Proposition \ref{prop:kunneth}, we have
\[
\h^1 \left( \Gga, M \right) 
\cong \bigoplus_{i=1}^\ell \left( \h^1 \left( \Gga, \Gev{\sfm_i^{n_i}} M_i \right) \otimes \bigotimes_{j\neq i} \hom_\Gga \left( \C, \Gev{\sfm_j^{n_j}} M_j \right) \right).
\]
Now, notice that $\hom_\Gga \left( \C, \Gev{\sfm_k^{n_k}} M_k \right) \cong  \hom_{\ga/\sfm_k^{n_k}} \left( \C, M_k \right)$ for all $k \in \{1, \dotsc, \ell\}$.  
This proves part~\eqref{prop:h1(ga,M)c}.  If $\hom_{\lie g \otimes A/\sfm_k^{n_k}} (\C, M_k) = 0$ for more than one index $k$, then for each $i$, there exists $j \neq i$ such that $\hom_{\lie g \otimes A/\sfm_j^{n_j}} (\C, M_j) = 0$.  This proves part~\eqref{prop:h1(ga,M)a}.  If $\hom_{\lie g \otimes A/\sfm_i^{n_i}} (\C, M_i) = 0$ for exactly one index $i$, then $\bigotimes_{j \neq k} \hom_{\lie g \otimes A/\sfm_j^{n_j}} (\C, M_j) = 0$ for all $k \ne i$.  This proves part~\eqref{prop:h1(ga,M)b}.
\end{proof}

The next result generalizes \cite[Lemma 3.3]{kodera10} and \cite[Proposition~3.6]{NS15}.

\begin{corollary} \label{cor:Ext1=0}
Let $V$ and $V'$ be nontrivial, finite-dimensional, irreducible $\Gga$-modules.  If the supports of $V$ and $V'$ are disjoint, then $\Ext^1_{\Gga} (V, V') = 0$.
\end{corollary}
\begin{proof}
This proof follows directly from Corollary~\ref{cor:ext=h.disjoint.supp}, the classification of finite-dimensional irreducible $\Gga$-modules given in Section~\ref{sec:periplectic} and Proposition~\ref{prop:h1(ga,M)}\eqref{prop:h1(ga,M)a}.
\end{proof}

Recall that we are assuming that $\lie g$ is a finite-dimensional simple Lie superalgebra, $A$ is an associative, commutative, finitely-generated algebra with unit, and $\Gamma$ is a finite abelian group acting on $\g$ and $A$ by automorphisms, such that the induced action of $\Gamma$ on $\maxspec (A)$ is free.  Now we state, and prove, the main result of this section.  It describes 1-extensions between finite-dimensional irreducible $\Gga$-modules in terms of homomorphisms and extensions between finite-dimensional $\ga/\sfm^n$-modules.

\begin{theorem} \label{thm:main}
Let $\pi, \pi' \in \cal P$, $\ell, n_1, \dotsc, n_\ell \in \bb Z_{>0}$, $\sfm_1, \dotsc, \sfm_\ell \subseteq A$ be maximal ideals in distinct $\Gamma$-orbits, and $V_i, V'_i$ be finite-dimensional irreducible $\lie g \otimes A/\sfm_i^{n_i}$-modules such that $\cal V (\pi) = \bightimes_{i=1}^\ell \Gev{\sfm_i^{n_i}} V_i$ and $\cal V (\pi') = \bightimes_{i=1}^\ell \Gev{\sfm_i^{n_i}} V_i'$.  For each $i \in \{1, \dotsc, \ell\}$, let $d_i$ denote $\delta_{1, n_i} \dim_{A/\sfm_i} \sfm_i/\sfm_i^2$.

\begin{enumerate}[$(a)$]
\item \label{thm:main.a}
If $V_i$ is not isomorphic to $V_i'$ for two or more indices $i$, then $\Ext_{\Gga}^1(\cal V (\pi), \cal V (\pi')) = 0$.

\item \label{thm:main.b}
If $V_i$ is isomorphic to $V_i'$ for all but one index $i$, then
\[
\Ext_{\Gga}^1 (\cal V (\pi), \cal V (\pi'))^{\oplus 2^{\kappa(\pi) + \kappa(\pi')}}
\cong \Ext^1_{\lie g \otimes A/\sfm_i^{n_i}} (V_i, V'_i)
\oplus \hom_\g \left( \g \otimes V_i, V_i' \right)^{\oplus d_i}.
\]

\item \label{thm:main.c}
If $V_i$ is isomorphic to $V_i'$ for all $i \in \{1, \dotsc, \ell\}$, then
\[
\Ext_{\Gga}^1 (\cal V (\pi), \cal V (\pi'))^{\oplus 2^{\kappa(\pi) + \kappa(\pi')}}
\cong \bigoplus_{i=1}^\ell  \left( \Ext^1_{\lie g \otimes A/\sfm_i^{n_i}} (V_i, V'_i) \oplus \hom_\g \left( \g \otimes V_i, V_i' \right)^{\oplus d_i} \right).
\]
\end{enumerate}
\end{theorem}

\begin{proof}
Denote $\left( \bigotimes_{i=1}^\ell \Gev{\sfm_i^{n_i}} V_i \right)$ by $V$, $\left( \bigotimes_{i=1}^\ell \Gev{\sfm_i^{n_i}} V'_i \right)$ by $V'$, and recall from Section \ref{ss:map.superalgebras} that $V \cong \cal V (\pi)^{\oplus 2^{\kappa(\pi)}}$ and $V' \cong \cal V(\pi')^{\oplus 2^{\kappa(\pi')}}$.  Thus, by Lemma~\ref{lem:isos}, we have
\[
\Ext^1_\Gga (\cal V (\pi), \cal V (\pi'))^{\oplus 2^{\kappa(\pi) + \kappa(\pi')}}
\cong \h^1 \left( \Gga, V^* \otimes V' \right).
\]
Now, notice that $V^* \otimes V' \cong \bigotimes_{i=1}^\ell \Gev{\sfm_i^{n_i}} (V_i^* \otimes V'_i)$, that $V_i^* \otimes V_i'$ are finite dimensional, and that 
\[
\hom_{\ga / \sfm_i^{n_i}} (\C, V_i^* \otimes V'_i) \cong
\hom_{\ga / \sfm_i^{n_i}} (V_i, V'_i) \cong
\begin{cases}
0, & \textup{if $V_i \not\cong V'_i$,} \\
\C, & \textup{if $V_i \cong V'_i$.}
\end{cases}
\]
\begin{enumerate}[(a)]\itemsep1ex
\item To prove part~\eqref{thm:main.a}, notice that, if $V_i \not\cong V_i'$ for two or more indices $i$, then by Proposition~\ref{prop:h1(ga,M)}\eqref{prop:h1(ga,M)a},
\[
\Ext^1_\Gga (\cal V (\pi), \cal V (\pi')) = 0.
\]

\item If $V_i \not\cong V_i'$ for exactly one index $i$, then by Proposition~\ref{prop:h1(ga,M)}\eqref{prop:h1(ga,M)b}, we have
\[
\Ext^1_\Gga (\cal V (\pi), \cal V (\pi'))^{\oplus 2^{\kappa(\pi) + \kappa(\pi')}}
\cong \h^1 \left( \Gga, \Gev{\sfm_i^{n_i}} (V_i^* \otimes V_i') \right).
\]
Now, let $I = \prod_{\gamma \in \Gamma} (\gamma \sfm_i)^{n_i}$.  By Lemma~\ref{lem:h1(ga,M)}, we have
\[
\Ext^1_\Gga (\cal V (\pi), \cal V (\pi'))^{\oplus 2^{\kappa(\pi) + \kappa(\pi')}}
\cong \h^1 \left( \ga/\sfm_i^{n_i}, (V_i^* \otimes V_i') \right) \oplus K_i,
\]
where $K_i$ is the kernel of the transgression map
\[
t_{_{V_i^* \otimes V_i}} \colon
\hom_{(\lie g \otimes A/I)^\Gamma} \left( (\lie g \otimes I / I^2)^\Gamma, \Gev{\sfm_i^{n_i}} (V_i^* \otimes V_i') \right)
\to \h^2 \left( (\ga/I)^\Gamma, \Gev{\sfm_i^{n_i}} (V_i^* \otimes V_i')  \right).
\]
To finish the proof of part~\eqref{thm:main.b}, notice that, by Proposition~\ref{prop:ker.transgression}, $K_i \cong \hom_\g \left( \g \otimes V_i, V_i' \right)^{\oplus d_i}$.

\item If $V_i \cong V_i'$ for all $i \in \{1, \dots, \ell\}$, then by Proposition~\ref{prop:h1(ga,M)}\eqref{prop:h1(ga,M)c}, we have
\[
\Ext^1_\Gga (\cal V (\pi), \cal V (\pi'))^{\oplus 2^{\kappa(\pi) + \kappa(\pi')}}
\cong \bigoplus_{i=1}^\ell \h^1 \left( \Gga, \Gev{\sfm_i^{n_i}} (V_i^* \otimes V_i') \right).
\]
The rest of the proof of part~\eqref{thm:main.c} follows from Lemma~\ref{lem:h1(ga,M)} and Proposition~\ref{prop:ker.transgression} using, for each $i \in \{1, \dotsc, \ell\}$, the same arguments that we used to prove part~\eqref{thm:main.b}.
\qedhere
\end{enumerate}
\end{proof}

Theorem~\ref{thm:main} generalizes \cite[Theorem~3.7]{NS13} to the super setting.  The particular case where the irreducible module is an evaluation one is treated in the following example.

\begin{example}[Evaluation modules] \label{eg:blocks.ev}
Let $\cal V(\pi) = \bightimes_{i=1}^\ell \Gev{\sfm_i} V (\lambda_i)$ and $\cal V(\pi') = \bightimes_{i=1}^\ell \Gev{\sfm_i} V(\mu_i)$ be irreducible finite-dimensional $\Gga$-modules.  In this case, $d_i = \dim_{A/\sfm_i} \sfm_i / \sfm_i^2$ and Theorem~\ref{thm:main} yields:
\begin{enumerate}[$(i)$]
\item $\Ext_{\Gga}^1(\cal V (\pi), \cal V (\pi')) = 0$, when $\lambda_i \ne \mu_i$ for two or more indices $i$.

\item $\Ext_{\Gga}^1 (\cal V (\pi), \cal V (\pi'))^{\oplus 2^{\kappa(\pi) + \kappa(\pi')}} \cong \Ext^1_\g (V(\lambda_i), V(\mu_i)) \oplus \hom_\g \left( \g \otimes V(\lambda_i), V(\mu_i) \right)^{\oplus d_i}$, when $\lambda_i = \mu_i$ for all but one index $i$.

\item $\Ext_{\Gga}^1 (\cal V (\pi), \cal V (\pi'))^{\oplus 2^{\kappa(\pi) + \kappa(\pi')}} \cong \bigoplus_{i=1}^\ell \Ext^1_\g (V(\lambda_i), V(\mu_i)) \oplus \hom_\g \left( \g \otimes V(\lambda_i), V(\mu_i) \right)^{\oplus d_i}$, when $\lambda_i = \mu_i$ for all $i \in \{1, \dotsc, \ell\}$.
\end{enumerate}
Thus, $\cal V(\pi)$ and $\cal V(\pi')$ are in the same block if and only if, for each $i \in \{1, \dots, \ell\}$, either: $\lambda_i - \mu_i \in Q$, or $V(\lambda_i)$ and $V(\mu_i)$ are in the same block in the category of finite-dimensional $\g$-modules.  In particular, if $\g$ is of type II, $\lie p(n)$, $\lie q(n)$, $S(n)$, $\tilde S(n)$ or $H(n)$, since all finite-dimensional irreducible $\Gga$-modules are evaluation modules, this gives a description of the block decomposition of the category of finite-dimensional $\Gga$-modules.

Moreover, if either $\lambda$ or $\mu$ are \emph{typical}, then $\Ext^1_\g (V(\lambda), V(\mu)) = 0$ (see \cite[Theorem 1]{kac78}).  Thus, if either $\lambda_i$ or $\mu_i$ is \emph{typical} for each $i \in \{ 1, \dotsc, \ell \}$, then $\Ext_{\Gga}^1 (\cal V (\pi), \cal V (\pi'))$ depends only on $\hom_\g \left( \g \otimes V(\lambda_i), V(\mu_i) \right)$.  In particular, if $\g$ is of type $B(0,n)$, since all the weights are typical, the blocks of the category of finite-dimensional $\Gga$-modules are parametrized by the so-called spectral character (compare with \cite[Proposition 4.5]{kodera10} and \cite[Theorem 5.19]{NS15}).

Now, let $\g$ be a Lie superalgebra of type $A$ or $C$ (type I), $\C$ be the trivial $\g$-module and $V(\lambda)$ be an irreducible finite-dimensional $\g$-module of highest weight $\lambda\in \lie h^*$.  Let $\rho_1$ be the half sum of the odd positive roots of $\g$, $\alpha_{\max}$ (resp. $\alpha_{\min}$) be the maximal (resp. minimal) root of $\g$, and $\mu^{(j)}$ be as in \cite[Theorem~1.1]{SZ07}.  For any maximal ideal $\sfm\in\maxspec A$, we have
	\[
\Ext_{\Gga}^1 (\Gev{\sfm} \C, \Gev{\sfm} V(\lambda)) \cong \Ext^1_{\lie g} (\C, V(\lambda)) \oplus \hom_\g \left( \g , V(\lambda) \right)^{\oplus d},
	\]
where $d=\dim_{A/\sfm} \sfm/\sfm^2$.  Thus $\Ext_{\Gga}^1 (\Gev{\sfm} \C, \Gev{\sfm} V(\lambda))\neq 0$ if and only if $\Ext^1_{\lie g} (\C, V(\lambda)) \ne 0$ or $\hom_\g \left( \g , V(\lambda) \right) \ne 0$.  From \cite[Theorem~1.1]{SZ07}, we obtain that:
\begin{enumerate}[$\bullet$]
\item in type $A$, $\Ext_{\Gga}^1 (\Gev{\sfm} \C, \Gev{\sfm} V(\lambda))\neq 0$ if and only if $\lambda \in \{-\alpha_{\min}, \alpha_{\max}, \mu^{(0)}, \dotsc, \mu^{(n-1)} \}$,
\item in type $C$, $\Ext_{\Gga}^1 (\Gev{\sfm} \C, \Gev{\sfm} V(\lambda))\neq 0$ if and only if $\lambda \in \{-\alpha_{\min},\ 2\rho_1,\ \alpha_{\max}\}$.
\end{enumerate}
\end{example}

\begin{remark}
In the published version of this paper, Example~5.9 was incorrect (see \cite[Remark~3.6]{CM}).
\end{remark}

\end{document}